\documentclass[a4paper, 10pt]{article}
\usepackage[utf8]{inputenc}
\usepackage[left=20 mm, bottom=20 mm, right=20 mm, top=20 mm]{geometry}
\usepackage{amsmath, amssymb, amsthm, csquotes, tikz, float}
\usepackage{hyperref}

\usepackage[english]{babel}
\usepackage{graphicx,epstopdf,epsfig}
\usepackage{amsfonts,epsfig,fancyhdr, graphics, hyperref,amsmath, tikz,amssymb}
\usepackage{geometry}
\newtheorem{theorem}{Theorem}[section]
\newtheorem{lemma}{Lemma}[section]

\newtheorem{definition}{Definition}[section]
\newtheorem{cor}{Corollary}[section]
\newtheorem{eg}{Example}[section]
\numberwithin{equation}{section}

\newcommand{\Names}{Partha Rana, Sriparna Bandopadhyay}
\newcommand{\Title}{Irreducible Combinatorially Symmetric Sign Patterns Requiring a Unique Inertia}

\def\det{{\rm det}}
\def\sgn{{\rm sgn}}
\def\In{{\rm In}}
\def\Ch{{\rm Ch}}
\def\min{{\rm min}}
\def\max{{\rm max}}
\def\RIn{{\rm RIn}}
\def\dist{{\rm dist}}

\begin{document}	
	\title{\Title\thanks{
			Corresponding Author: Partha Rana.}
		\author{Partha\ Rana\thanks{Department of Mathematics, Indian Institute of Technology Guwahati, Guwahati,
				Assam 781039, India.\\
				E-mail addresses: r.partha@iitg.ac.in (P. Rana), sriparna@iitg.ac.in (S. Bandopadhyay).}
			\and Sriparna Bandopadhyay}
	}

	\markboth{\Names}{\Title}

	\date{}
	
	\maketitle

	\begin{abstract}
		A sign pattern is a matrix whose entries belong to the set $\{+,-,0\}$. A sign pattern requires a unique inertia if every real matrix in its qualitative class has the same inertia. Symmetric tree sign patterns requiring a unique inertia has been studied extensively in \cite{2001, 2001a, 2018}. Necessary and sufficient conditions in terms of the symmetric minimal and maximal rank, as well as conditions depending on the position and sign of the loops in the underlying graph of such patterns has been used to characterize inertia of symmetric tree sign patterns. In this paper, we consider combinatorially symmetric sign patterns with a $0$-diagonal and identify some such patterns with interesting combinatorial properties, which does not require a unique inertia. Initially, we begin with combinatorially symmetric tree sign patterns with a $0$-diagonal, with a special focus on tridiagonal sign patterns. We then consider patterns whose underlying undirected graph contain cycles but no loops, and we derive necessary conditions based on the sign of the edges and the distance between the cycles in the underlying graph for such patterns to require a unique inertia.
	\end{abstract}
	
	\noindent {\bf Key words:} Sign pattern, Inertia, Tridiagonal sign pattern.
	
	\noindent {\bf AMS Subject Classification:} 05C50, 15A18, 15B35.
	\maketitle



	
	\section{Introduction}
	
	An $n\times n$ matrix $\mathcal{P}=[p_{ij}]$ whose entries belong to the set $\{+,-,0\}$ is called a \textit{sign pattern matrix} or a \textit{sign pattern}. The set of all $n\times n$ sign patterns is denoted by $\mathcal{Q}_n$ and the set of all real matrices $$\mathcal{Q}(\mathcal{P})=\{\mathbf{A}=[a_{ij}]\in\mathbb{R}^{n\times n}: \sgn(a_{ij})=p_{ij}~ \text{for} ~\text{all}~ i,j=1,2,...,n\}$$ is called the \textit{qualitative class} of the sign pattern $\mathcal{P}\in \mathcal{Q}_n$.
	
	$(+)+(-)$ is called an \textit{ambiguous entry} and is denoted by \#. Let $\mathcal{P}_1=[p^1_{ij}]$, $\mathcal{P}_2=[p^2_{ij}]$ be two sign patterns of order $n$. Then $\mathcal{P}_1+\mathcal{P}_2$ is defined unambiguously if $p^1_{ij}p^2_{ij}\neq -$ for all $i,j$ and the product $\mathcal{P}_1 \mathcal{P}_2$ is defined unambiguously if for all $i,j$; $\sum_{k=1}^{n}{p^1_{ik}p^2_{kj}}$ does not contain oppositely signed terms.
	
	A sign pattern $\mathcal{P}\in \mathcal{Q}_n$ is said to be \textit{sign nonsingular} if every $A\in \mathcal{Q}(\mathcal{P})$ is nonsingular. Thus, $\mathcal{P}$ is sign nonsingular if and only if $\det(\mathcal{P})=+$ or $\det(\mathcal{P})=-$, that is, in the standard expansion of $\det(\mathcal{P})$ into $n!$ terms, there is at least one nonzero term, and all nonzero terms have the same sign.
	If all matrices $A\in \mathcal{Q(\mathcal{P})}$ are singular, or equivalently, if $\det(\mathcal{P})=0$, that is, all the terms in the expansion of $\det(\mathcal{P})$ are equal to zero, then $\mathcal{P}$ is said to be \textit{sign singular}.

	The \textit{signed directed graph} $D$ of an $n\times n$ sign pattern matrix $\mathcal{P}=[p_{ij}]$, is the directed graph having n vertices $\{1,2,...,n\}$, such that there is a directed edge in $D$ from $i$ to $j$, denoted by $(i,j)$, if and only if $p_{ij}\neq 0$. The arc $(i,j)$ is associated with the sign $+~\text{or}~-$ if and only if $p_{ij}=+~\text{or}~-$. 
	A product of the form $\gamma=p_{i_1i_2}p_{i_2i_3}\cdots p_{i_ki_{k+1}}$, where all elements are nonzero and $\{i_1,i_2,...,i_k, i_{k+1}\}$ consists of distinct indices, is called a \textit{path} of length $l(\gamma)=k$ from $i_1$ to $i_{k+1}$.
	If $i_1 = i_{k+1}$ and $\{i_1, i_2, \dots, i_k\}$ consists of distinct indices, then $\gamma$ is called a \textit{simple cycle}  of length $l(\gamma)=k$, where $\gamma$ is denoted by either $(i_1i_2\cdots i_k)$ or $p_{i_1i_2}p_{i_2i_3}\cdots p_{i_ki_{1}}$. For $k=1$, the simple cycle $\gamma=p_{i_1i_1}$ is called a \textit{loop}.
	A simple cycle $\gamma$ is said to be positive (respectively, negative) if $\sgn(\gamma)=(-)^{k-1}p_{i_1i_2}p_{i_2i_3}\cdots p_{i_ki_{1}}$ is positive (respectively, negative). 
	Suppose that $\gamma_i$ are mutually vertex-disjoint simple cycles with 
	$l(\gamma_i) = k_i$ for $i = 1, 2, \dots, t$ then the product of simple cycles, $\Gamma=\gamma_1 \gamma_2\cdots \gamma_t$ is called a \textit{composite cycle} of length $l(\Gamma)=\sum_{i=1}^{t}k_i$ and $\Gamma$ is said to be positive (respectively, negative) if $\sgn(\Gamma)=\prod_{i=1}^t\sgn(\gamma_i)$ is positive (respectively, negative). 
	Suppose that $\Gamma_1 = \gamma_1\gamma_2\cdots\gamma_t$ is a composite cycle in $D$ and let $\Gamma_2$ be another composite cycle in $D$. We define
	$$
	\Gamma_1 \setminus \Gamma_2 = \prod\{\gamma_i : \gamma_i \in \Gamma_1 ~\text{and } \gamma_i \notin \Gamma_2\},$$
	where \(\gamma_i \text{ in } \Gamma_1\) is a simple cycle in $D$.
	If $S$ is a set of vertices in $D$, then $D \setminus S$ denotes the subgraph of $ D $ obtained by deleting the vertices in $S$ along with their incident edges.

	Let $\mathcal{P}\in \mathcal{Q}_n$ be a \textit{combinatorially symmetric} sign pattern, that is, $p_{ij}\neq 0$ if and only if $p_{ji}\neq 0$.
	The \textit{signed undirected graph} of a combinatorially symmetric sign pattern $\mathcal{P}$ is the undirected graph $G$, whose edges are signed, so that any edge $\{i,j\}$ is $+$ (respectively, $-$) if $p_{ij}p_{ji} = +$
	(respectively, $p_{ij}p_{ji}=-$). 
	A sequence of edges $P=\{i_1, i_2\} \{i_2, i_3\} \cdots \{i_k, i_{k+1}\}$ in $G$, where the vertices $i_1, i_2, \dots, i_{k+1}$ are all distinct, is called a \textit{path} of length $l(P)=k$.
	For $k \geq 3$, if $i_1 = i_{k+1}$ and the vertices $i_1, i_2, \dots, i_k$ are all distinct, then the closed path of the form $\{i_1, i_2\} \{i_2, i_3\} \cdots \{i_k, i_1\}$ is called a \textit{simple cycle}, denoted by $\mathcal{C} = i_1i_2\cdots i_k$, with length $l(\mathcal{C}) = k$.
	A product of mutually vertex-disjoint simple cycles in $G$, denoted by $\mathcal{C} = \mathcal{C}_1 \mathcal{C}_2 \cdots \mathcal{C}_t$, is called a \textit{composite cycle} of length $l(\mathcal{C})=\sum_{i=1}^{t} k_i$, where $k_i$ is the length of the $i$-th simple cycle. A set of edges $\mathcal{M} = \{\{i_1, j_1\}, \{i_2, j_2\}, \dots, \{i_k, j_k\}\}$ in $G$ is called a \textit{matching} if no two edges in $\mathcal{M}$ are adjacent. The \textit{length} of $\mathcal{M}$, denoted by $l(\mathcal{M})$, is defined as the number of edges in $\mathcal{M}$.
	For a vertex $i$ of $G$, the \textit{degree} of $i$, denoted by $\deg(i)$, is the number of edges of $G$ incident to $i$.
	If $S$ is a set of vertices in $G$, then $G \setminus S$ denotes the subgraph of $ G $ obtained by deleting the vertices in $S$ along with their incident edges.
	All the patterns throughout this paper are assumed to be irreducible combinatorially symmetric with a $0$-diagonal.
	If the underlying signed undirected graph $G$ is a tree (respectively, path), then $\mathcal{P}$ is referred to as a \textit{tree sign pattern} (respectively, \textit{path sign pattern}). A path sign pattern is also referred to as a \textit{tridiagonal sign pattern}.
	

	
	Suppose that $\mathcal{P} $ is an irreducible sign pattern whose underlying signed undirected graph is $G$, with the vertex set denoted by $V(G)$.
	The distance between two vertices $u,v\in V(G)$, denoted by $\dist(u,v)$, is defined as the minimum length among all directed paths from $u$ to $v$. Furthermore, suppose that $\mathcal{C}=u_1u_2\cdots u_k$ is a cycle in $G$. Then, the distance from a vertex $u\in V(G)$ to the cycle $\mathcal{C}$ is defined by $\dist(u,\mathcal{C})=\min\{\dist(u,u_i): i=1,2,...,k\}.$

	
	Let $\mathcal{P} \in \mathcal{Q}_n$ be an irreducible combinatorially symmetric sign pattern with a $0$-diagonal whose underlying signed undirected graph is $G$. We say that the path $P$ is a \textit{maximal signed positive path} in $G$  if $P$ satisfies the following: 
	\begin{itemize}
		\item[i.] the path starts with the first positive edge, 
		or with the first positive edge that follows a negative edge
		\item[ii.] contains successive 
		positive edges and
		\item[iii.] ends at the last positive edge, or when a negative edge occurs. 
	\end{itemize}
	Similarly, we define a \textit{maximal signed negative path}.  
	For example, let $\mathcal{P}=[p_{ij}]$ be an irreducible combinatorially symmetric sign pattern whose underlying signed undirected graph $G$ is given in Fig. \ref{nfig11}. 
	
	\begin{figure}[H] \label{nfig11}
		\tikzset{node distance=1cm}
		\centering
		\begin{tikzpicture}
			\tikzstyle{vertex}=[draw, circle, inner sep=0pt, minimum size=.15cm, fill=black]
			\tikzstyle{edge}=[,thick]
			\node[vertex,label=right:1](v1)at(2,0){};
			\node[vertex,label=above:2](v2)at(1,1.73){};
			\node[vertex,label=above:3](v3)at(-1,1.73){};
			\node[vertex,label=left:4](v4)at(-2,0){};
			\node[vertex,label=below:5](v5)at(-1,-1.73){};
			\node[vertex,label=below:6](v6)at(1,-1.73){};
			\draw[edge](v1)--node[right, yshift=0cm, xshift=-0cm]{$-$}(v2);
			\draw[edge](v2)--node[above]{$- $}(v3);
			\draw[edge](v3)--node[left, yshift=-0cm, xshift=-0.1cm]{$-$}(v4);
			\draw[edge](v4)--node[left, yshift=-0.1cm, xshift=-0.1cm]{$+$}(v5);
			\draw[edge](v5)--node[below, yshift=-0cm, xshift=0cm]{$+$}(v6);
			\draw[edge](v6)--node[right, yshift=-0.1cm, xshift=0cm]{$+$}(v1);
		\end{tikzpicture} \caption{G}\end{figure}
	Then the path $P_1=\{1,2\}\{2,3\}\{4,3\}$ is a maximal signed negative path and $P_2=\{4,5\}\{5,6\}\{6,1\}$ is a maximal signed positive path.
	
	Suppose $P$ is a matrix property that a real matrix may or may not have. Then a sign pattern $\mathcal{P}$ is said to \textit{require} $P$ if every real matrix in $\mathcal{Q}(\mathcal{P})$ satisfies the property $P$, or to \textit{allow} $P$ if some real matrix in $\mathcal{Q}(\mathcal{P})$ satisfies $P$.

	Let $A$ be an $n\times n$ real matrix, the \textit{inertia} of $A$,  denoted by $\In(A)$, is the triple of nonnegative numbers $\In(A)=(i_+(A), i_-(A),i_0(A))$, where $i_+(A)$, $i_-(A)$ and $i_0(A)$ denote the number of eigenvalues of $A$ with positive, negative and zero real parts, respectively. Furthermore, $\RIn(A) = (i_+(A), i_-(A), i_z(A), 2i_{p}(A))$ is the \textit{refined inertia} of $A$, where $i_z(A)$, $2i_{p}(A)$ represent the number of zero eigenvalues, nonzero purely imaginary eigenvalues of A, respectively. The \textit{eigenvalue frequency} of an $n\times n$ real matrix $A$ is defined as the ordered pair $S_A=(k,n-k)$, where $k$ is the number of real eigenvalues of $A$ and $n-k$ is the number of nonreal eigenvalues of $A$.

	For a sign pattern $\mathcal{P} \in \mathcal{Q}_n$, the inertia of $\mathcal{P}$ is defined as $\In(\mathcal{P})=\{\In(A): A\in \mathcal{Q}(\mathcal{P})\}$ and the refined
	inertia of $\mathcal{P}$ is defined as $\RIn(\mathcal{P})=\{\RIn(A): A\in \mathcal{Q}(\mathcal{P})\}$. The sign pattern $\mathcal{P}$ is called $k$-consistent, for some fixed nonnegative integer $k$ with $0\leq k \leq n$, if $S_A=(k,n-k)$ for all $A\in \mathcal{Q}(\mathcal{P})$. If $\mathcal{P}$ is $k$-consistent, then we write $S_{\mathcal{P}}=(k,n-k)$. Additionally, $\mathcal{P}$ is said to be consistent if it is $k$-consistent for some integer $k$, where $0\leq k\leq n$.
	
	Let $A$ be a square matrix of order $n$. If $\alpha$, $\beta$ are subsets of $\{1,2,...,n\}$, then $A[\alpha,\beta]$ denotes the submatrix of $A$ with rows, columns corresponding to the indices in $\alpha, \beta$, respectively. In particular, $A[\alpha]$ denotes $A[\alpha, \alpha]$.

	
	If $\mathcal{P}_1$, $\mathcal{P}_2$ are two sign patterns, then the patterns are said to be equivalent if one can be obtained from
	the other through a series of permutation similarity, signature similarity, negation, and transposition. Being
	equivalent, $\mathcal{P}_1$ requires a unique inertia if and only if $\mathcal{P}_2$ requires a unique inertia.
	
	Let $\mathcal{P}_1, \mathcal{P}_2 \in \mathcal{Q}_n$ be two sign patterns. We say that $\mathcal{P}_2$ is a \textit{superpattern} of $\mathcal{P}_1$ if $\mathcal{P}_2$ can be obtained from $\mathcal{P}_1$ by replacing some (possibly none) of the zero entries of $\mathcal{P}_1$ with either $+$ or $-$. In this case, $\mathcal{P}_1$ is called a \textit{subpattern} of $\mathcal{P}_2$.

	Suppose $\mathcal{P}=(p_{ij})\in \mathcal{Q}_n$ is an irreducible sign pattern whose underlying signed directed graph $D$ contains a simple $k$-cycle $\gamma$.  Define the real matrix $B_\gamma(0)=(b_\gamma(0)_{ij})$ by
	\begin{equation} \label{e1}
		b_{\gamma}(0)_{ij}= \begin{cases}
			1 & \text{if}\ p_{ij}=+ ~\text{and}~ \text{is}~ \text{in}~ \gamma \\
			-1 & \text{if}\ p_{ij}=- ~\text{and}~ \text{is}~ \text{in}~ \gamma \\
			0 & \text{elsewhere,}
		\end{cases}
	\end{equation}
	and define the perturbed matrix $B_\gamma(\epsilon)=(b_\gamma(\epsilon)_{ij})$ by
	\begin{equation} \label{e2}
		b_{\gamma}(\epsilon)_{ij}= \begin{cases}
			b_\gamma(0)_{ij} & \text{if}\ p_{ij}~ \text{is}~ \text{in}~ \gamma \\
			\epsilon & \text{if}\ p_{ij}=+ ~\text{and}~ \text{is}~\text{not}~ \text{in}~ \gamma \\
			-\epsilon & \text{if}\ p_{ij}=- ~\text{and}~ \text{is}~ \text{not}~\text{in}~ \gamma \\
			0 & \text{elsewhere,}
		\end{cases}
	\end{equation}
	for some $\epsilon>0$. Since $B_\gamma(0)$ has $k$ nonzero algebraically simple eigenvalues which are the $k$-th complex roots of $+1$, $-1$ depending on $\sgn(\gamma)$, and the eigenvalues of a real matrix are continuous functions of the entries of the matrix, for $\epsilon>0$ sufficiently small, the perturbed matrix $B_\gamma(\epsilon)\in \mathcal{Q}(\mathcal{P})$ has $k$ algebraically simple nonzero eigenvalues close to the nonzero eigenvalues of $B_\gamma(0)$. If $\gamma$ is a negative even $k$-cycle, then $B_\gamma(0)$ has $k$ nonreal algebraically simple eigenvalues, thus, for $\epsilon>0$, sufficiently small the matrix $B_\gamma(\epsilon)$ has $k$ distinct nonreal eigenvalues. Similarly, if $\gamma$ is a positive even $k$-cycle then $B_\gamma(0)$ has algebraically simple real eigenvalues $1$ and $-1$, so for $\epsilon>0$, sufficiently small $B_\gamma(\epsilon)$ has two algebraically simple real eigenvalues close to $1$ and $-1$.
	Similarly if $\Gamma$ is a composite cycle in $D$, we can define $B_\Gamma(0)$ and $B_{\Gamma}(\epsilon)$ satisfying similar properties.

	
	Assume $\mathcal{P}\in \mathcal{Q}_n$ is a sign pattern such that the maximal cycle length in the signed directed graph of $\mathcal{P}$ is $m$. Let $B\in \mathcal{Q}(\mathcal{P})$ then the characteristic polynomial of $B$ (denoted by $\Ch_B(x)$) is given by $$\Ch_B(x)=x^n-E_1(B)x^{n-1}+E_2(B)x^{n-2}-\cdots+(-1)^mE_m(B)x^{n-m},$$ where $E_k(B)$ for $1\leq k\leq m$, is the sum of all cycles (simple or composite) of length $k$ in $B$ properly signed.
	Let $V_+(x)$ be the variation of signs in the terms of $\Ch_B(x)$ for $x>0$, and let $V_-(x)$ be the variation of signs in the terms of $\Ch_B(x)$ for $x<0$. Then, by Descartes' rule of signs, the number of positive (negative) roots of $\Ch_B(x)$ is either equal to $V_+(x)$ ($V_-(x)$, respectively) or is less than it by an even number. For example, consider $\Ch_B(x)=a_3x^3+a_2x^2+a_1x+a_0$, for some real matrix $B$ and $a_i\in \mathbb{R}$, $i=0,1,2,3$. Then, 
	$$
	\sgn(\Ch_B(x)) = \begin{cases} 
		(\sgn(a_3))(+)+(\sgn(a_2))(+)+(\sgn(a_1))(+)+(\sgn(a_0)) & \text{if} ~ x> 0 \\
		(\sgn(a_3))(-)+(\sgn(a_2))(+)+(\sgn(a_1))(-)+(\sgn(a_0)) & \text{if} ~x<0,
	\end{cases}
	$$ where the notation $(\sgn(a_i))(\cdot)$ in the $i$-th position represents the product of $\sgn(a_i)$ and the sign of $x^i$ for $3\geq i\geq 0$, depending on whether $x>0$ or $x<0$.
	If \( V_+(x) = 1 \) and \( V_-(x) = 2 \) for some \( a_i \in \mathbb{R} \), \( i = 0, 1, 2, 3 \), then the number of positive roots of $\Ch_B(x)$ is exactly one, and the number of negative roots is either two or zero.

	Symmetric sign patterns requiring a unique inertia has been studied in \cite{2001, 2001a, 2018, 2024}. For instance, in \cite{2001}, the authors characterized symmetric sign patterns that require a unique inertia in terms of the symmetric minimal rank and symmetric maximal rank. Also, sufficient conditions for symmetric tree sign patterns to require a unique inertia based on the sign and position of the loops in the underlying graph were given in \cite{2018, 2024}. In \cite{2001a}, Hall and Li characterized all irreducible symmetric tridiagonal sign patterns requiring a unique inertia. Lin et al. \cite{2018} characterized all irreducible sign patterns of order 2 and 3 that require a unique inertia.

	In Section \ref{s2}, to begin with, we consider irreducible combinatorially symmetric tree sign patterns with a $0$-diagonal and obtain some necessary conditions for such patterns to require a unique inertia based on the multiplicity of eigenvalues with zero real part. Next, we shift our focus to a particular type of tree sign pattern, the tridiagonal sign patterns with a $0$-diagonal, and we obtain necessary conditions based on the number of maximal signed paths of odd length that are allowed by the sign pattern for requiring a unique inertia. We also identify certain patterns which cannot occur as subpatterns of tridiagonal sign patterns requiring a unique inertia. 
	In Section \ref{s4}, we investigate irreducible combinatorially symmetric sign patterns whose underlying graphs contain cycles but no loops. We derive necessary conditions for such sign patterns to require a unique inertia based on the number of negative edges and other combinatorial properties of the cycles in the underlying graphs.
	
	We conclude this paper with interesting results which follow from those obtained in this paper, but could not be included in this manuscript. In Section \ref{s5}, we also include problems which come as natural extensions of the results obtained in this paper and that pave the way for future research in this topic.

	\section{Tree and path sign patterns with a $0$-diagonal requiring a unique inertia} \label{s2}
	
	As mentioned earlier, all the patterns in this paper are assumed to be combinatorially symmetric.
	In this section, to begin with, we show that if a tree sign pattern with a $0$-diagonal has no repeated eigenvalues with zero real part, then it has a unique inertia. 
	We also establish that for tridiagonal sign patterns with a $0$-diagonal, if there are no repeated zero eigenvalues, then it requires a unique inertia.
	The relationship between tridiagonal patterns requiring a unique inertia with the consistency of certain associated patterns were useful in identifying certain combinatorial structures of such patterns requiring a unique inertia.

	\begin{lemma} \label{lem1}
		Suppose that $f(x)$ is a real polynomial of even degree $n$, consisting only of even powers of $x$. If $\lambda$ is a root of $f(x)$ with multiplicity $m$, then $-\lambda$ is also a root of $f(x)$ with multiplicity $m$.
	\end{lemma}
	\begin{proof}
		Since $\lambda$ is a root of $f(x)$ with multiplicity $m$, so $f(x)=(x-\lambda)^mg(x)$ where $g(x)$ is a polynomial of degree $n-m$ with $g(\lambda)\neq 0$. Since $f(x)$ consists only of even powers of $x$, $$f(x)=f(-x)=(-1)^m(x+\lambda)^mg(-x).$$
		Thus $(x+\lambda)^m$ is also a factor of $f(x)$, so $f(x)=(x-\lambda)^m(x+\lambda)^mg^{(1)}(x)$. Therefore, $-\lambda$ is a root of $f(x)$ with multiplicity at least $m$. If $-\lambda$ is a root of $f(x)$ with multiplicity greater than $m$, then similarly $\lambda$ is also a root of $f(x)$ with multiplicity greater than $m$, which is a contradiction. Therefore, if $\lambda$ is a root of $f(x)$ with multiplicity $m$, then $-\lambda$ is also a root of $f(x)$ with multiplicity $m$.
	\end{proof}
	
	\begin{cor} \label{xlem1}
		Suppose that $f(x)$ is a real polynomial of odd degree $n$, consisting only of odd powers of $x$. If $\lambda$ is a root of $f(x)$ with multiplicity $m$, then $-\lambda$ is also a root of $f(x)$ with multiplicity $m$.
	\end{cor}
	\begin{proof}
		Since $f(x)$ is a real polynomial of odd degree $n$, consisting only of odd powers of $x$, we have $f(x)=xg(x)$, where $g(x)$ is a real polynomial of even degree that consists only of even powers of $x$. Hence, by Lemma~\ref{lem1}, the result follows.
	\end{proof}

	Since for a tree sign pattern $\mathcal{P} \in \mathcal{Q}_n$ with a $0$-diagonal, the characteristic polynomial $\Ch_B(x)$ of $B \in \mathcal{Q}(\mathcal{P})$ consists only of even powers of $x$ when $n$ is even and only odd powers of $x$ when $n$ is odd, Lemma \ref{lem1} and its corollary give the following result.
	\begin{cor}\label{c1}
		Suppose that $\mathcal{P}\in \mathcal{Q}_n$ is a tree sign pattern with a $0$-diagonal and $B\in \mathcal{Q}(\mathcal{P})$. If $\lambda$ is an eigenvalue of $B$ with algebraic multiplicity $m$, then $-\lambda$ is also an eigenvalue of $B$ with algebraic multiplicity $m$.
	\end{cor}
	The following result is by Eschenbach et al. \cite{1}.
	\begin{lemma}[Lemma 1.5, \cite{1}] \label{l1}
		If an $n\times n$ sign pattern matrix $\mathcal{A}$ does not allow repeated real eigenvalues, then $\mathcal{A}$ is consistent.
	\end{lemma}

	The following is a similar sufficient condition for tree sign patterns with a $0$-diagonal to require a unique inertia.
	
	\begin{theorem} \label{l3.3ne}
		If $\mathcal{P}\in \mathcal{Q}_n$ is a tree sign pattern with a $0$-diagonal such that it does not allow repeated eigenvalues with zero real part, then $\mathcal{P}$ requires a unique inertia.
	\end{theorem}
	\begin{proof}
		Let $B_1,B_2\in \mathcal{Q}(\mathcal{P})$ be arbitrary. Define $$B(t)=(1-t)B_1+tB_2$$ and $\In(B(t))=(i_+(t),i_-(t),i_0(t))$, where $i_+(t)+i_-(t)+i_0(t)=n$. If $b_1=b_2$, then $(1-t)b_1+tb_2=b_1$ for all $t$. Also, if $b_1 \neq b_2$ and $\sgn(b_1)=\sgn(b_2)$, then $b_1+t(b_2-b_1)$ has the same sign as $b_1$ for all $t\in [0-\delta, 1+\delta]$, where $$0<\delta<\frac{1}{|b_2-b_1|}\min\{|b_1|,|b_2|\}.$$ Therefore, for sufficiently small $\delta>0$, $B(t)\in \mathcal{Q}(\mathcal{P})$ for all $t\in [0-\delta,1+\delta]$. Let $c\in [0,1]$ and $B(c)$ have $i_0(c)$ eigenvalues with zero real part, then by hypothesis, they are all distinct. Suppose that $\lambda_1, \lambda_2,...,\lambda_{i_0(c)}$ are the eigenvalues with zero real part and let the other eigenvalues of $B(c)$ be $\lambda_{i_0(c)+1}, \lambda_{i_0(c)+2},...,\lambda_n$. Consider 
		\begin{align*}
			\epsilon_1&=\min\left\{ \frac{|\lambda_i-\lambda_j|}{2}: 1\leq i<j\leq n~ \text{and} ~ \lambda_i\neq \lambda_j\right\},\\ \epsilon_2&=\min\{|Re(\lambda_j)|: i_0(c)+1\leq j\leq n\},\\ \epsilon_c&=\min\{\epsilon_1,\epsilon_2\}~ \text{and} ~ D_i=\{c\in \mathbb{C}: |x-\lambda_i|<\epsilon_c\}.
		\end{align*}
		Then we have $D_i\cap D_j=\emptyset$ whenever $\lambda_i\neq \lambda_j$. Since the eigenvalues of $B(t)$ are continuous functions of $t$, therefore there exists a $\delta_c>0$ such that for all $i$, $1\leq i\leq i_0(c)$, the disk $D_i$ contains exactly one eigenvalue of $B(t)$ and for each $j$, $i_0(c)+1\leq j\leq n$, if $m_j$ is the algebraic multiplicity of $\lambda_j$, then the disk $D_j$ contains exactly $m_j$ eigenvalues of $B(t)$ whenever $|t-c|<\delta_c$. Since $B(t)$ is a real matrix whose underlying signed undirected graph is a tree with no loop, by Corollary \ref{c1} if $\lambda$ is an eigenvalue of $B(t)$, then $-\lambda$ is also an eigenvalue of $B(t)$, so the eigenvalue of $B(t)$ in $D_i$, $1\leq i\leq i_0(c)$ has zero real part.
		Also, for $i_0(c)+1\leq i\leq n$, since $D_i$ does not intersect the imaginary axis, we conclude that the real part of each eigenvalue of $B(t)$ in $D_i$ has the same sign as $Re(\lambda_i)$. Thus, $\In(B(t))=\In(B(c))$, for all $t\in(c-\delta_c,c+\delta_c)$. Since $c\in [0,1]$ is arbitrary, $\{(c-\delta_c,c+\delta_c)\}_c$ forms an open cover of $[0,1]$, which has a finite subcover. Let $\{(c-\delta_c,c+\delta_c)\}_{c\in S}$ where $S=\{c_1,c_2,...,c_l\}$ be a finite subcover of $[0,1]$. Since $\In(B(t))$ is constant in each $(c-\delta_c,c+\delta_c)$, $c\in S$, $\In(B(t))$ is constant for all $t\in [0,1]$. Therefore, $\In(B_1)=\In(B_2)$ and $\mathcal{P}$ require a unique inertia.
	\end{proof}
	The above result is not true in general if the underlying signed undirected graph of the tree sign pattern $\mathcal{P}$ contains loops.
	\begin{eg}
		Let $$\mathcal{P}=\begin{bmatrix}
			+&+\\+&+
		\end{bmatrix}.$$ Then $\text{trace}(A)>0$ for all $A\in \mathcal{Q}(\mathcal{P})$, therefore $\mathcal{P}$ does not allow repeated eigenvalues with zero real part. However, $\mathcal{P}$ does not require a unique inertia, since $(2,0,0),(1,0,1)\in \In(\mathcal{P})$.   
	\end{eg}
	Also, the conclusion of Theorem \ref{l3.3ne} is not true if $\mathcal{P}$ is not a tree sign pattern, even if it has a $0$-diagonal.
	\begin{eg}\label{example2.6}
		Let $$\mathcal{P}=\begin{bmatrix}
			0&+&-\\-&0&+\\+&-&0
		\end{bmatrix}.$$ Suppose that $$B=\begin{bmatrix}
			0&a_1&-a_2\\-a_3&0&a_4\\a_5&-a_6&0
		\end{bmatrix}\in \mathcal{Q}(\mathcal{P}),$$
		where $a_i>0$ for $i=1,2,...,6$. Then $\Ch_B(x)=x^3+(a_1a_3+a_2a_5+a_4a_6)x+(a_2a_3a_6-a_1a_4a_5)$. Then we have the following cases.
		\begin{itemize}
			\item[i.] If $a_2a_3a_6-a_1a_4a_5>0$, then
			$$
			\sgn(\Ch_B(x)) = \begin{cases} 
				(+)+(+)(+)+(+) & \text{if} ~ x> 0 \\
				(-)+(+)(-)+(+) & \text{if} ~x<0.
			\end{cases}
			$$
			So $V_+(x)=0$ and $V_-(x)=1$, and by Descartes' rule of signs, $\Ch_B(x)$ has exactly one negative root and no positive or zero root.
			\item[ii.] If $a_2a_3a_6-a_1a_4a_5<0$, then
			$$
			\sgn(\Ch_B(x)) = \begin{cases} 
				(+)+(+)(+)+(-) & \text{if} ~ x> 0 \\
				(-)+(+)(-)+(-) & \text{if} ~x<0.
			\end{cases}
			$$
			So $V_+(x)=1$ and $V_-(x)=0$, and by Descartes' rule of signs, $\Ch_B(x)$ has exactly one positive root and no negative or zero root.
			\item[iii.] If $a_2a_3a_6-a_1a_4a_5=0$, then
			$\Ch_B(x)=x(x^2+(a_1a_3+a_2a_5+a_4a_6))$, thus $0$ is the only real root of $\Ch_B(x)$.
		\end{itemize} Hence, any $B\in \mathcal{Q}(\mathcal{P})$ has exactly one real eigenvalue. Therefore, $\mathcal{P}$ has all distinct eigenvalues and does not allow repeated eigenvalues with zero real part. However,
		\begin{itemize}
			\item[i.] if $a_i=1$ for all $i=1,2,...,6$, then $\In(B)=(0,0,3),$
			\item[ii.] if $a_1=2$ and $a_i=1$ for $i=2,3,...,6$, then $\In(B)=(1,2,0).$
		\end{itemize}
		So, $\mathcal{P}$ does not require a unique inertia. 
	\end{eg}

	The converse of Theorem \ref{l3.3ne} is not true for tree sign patterns with a $0$-diagonal, as the following example shows.
	\begin{eg}
		Let $$\mathcal{P}=\begin{bmatrix}
			0&+&+&+\\+&0&0&0\\+&0&0&0\\+&0&0&0
		\end{bmatrix}.$$ Then $\mathcal{P}$ is a sign symmetric tree sign pattern with zero diagonal, and by Remark~4.1\cite{2018}, every matrix $B\in\mathcal{Q}(\mathcal{P})$ is similar to a symmetric matrix, so by Theorem 4.3\cite{2001}, $\mathcal{P}$ requires a unique inertia. However, for $B\in \mathcal{Q(\mathcal{P})}$ with all nonzero entries equal to $1$, the eigenvalues of $B$ are $0,0,\sqrt{3},-\sqrt{3}$. Here, the eigenvalue $0$ is repeated.
	\end{eg}


	For symmetric sign patterns whose underlying graphs are trees, Hall et al. \cite{2001} obtained the following result.
	\begin{theorem}[Theorem 4.5, \cite{2001}]  Let $\mathcal{A}$ be a symmetric tree sign pattern, with the maximum length of the cycles in A equal to $m\geq 1$. Then $\mathcal{A}$ requires a unique inertia if and only if all the terms in $E_m(B)$ have the same sign for any $B \in \mathcal{Q}(\mathcal{A})$. In this case, $\mathcal{A}$ requires rank $m$.
	\end{theorem}

	We investigate the above result for general sign patterns which are not necessarily symmetric.
	
	\begin{lemma}\label{xnl1}
		Let $\mathcal{P}\in \mathcal{Q}_n$ be a sign pattern, whose underlying signed directed graph is $D$. Suppose that the maximum length of a composite cycle in $D$ is $m$, $m\geq 2$, then $\mathcal{P}$ requires a unique inertia only if all the composite cycles in $D$ of length $m$ have the same sign.
	\end{lemma}
	\begin{proof}
		Since the maximum length of a composite cycle in $D$ is $m$, the characteristic polynomial of any $B\in \mathcal{Q}(\mathcal{P})$ is given by, 
		$$\Ch_B(x)=x^n-E_1(B)x^{n-1}+E_2(B)x^{n-2}-\cdots+(-1)^mE_m(B)x^{n-m},$$
		where $E_k(B)$ for $1\leq k\leq m$ is the sum of all cycles (simple or composite) of length $k$ in $B$ properly signed. Suppose $\Gamma_1$, $\Gamma_2$ are two oppositely signed composite cycles in $D$ of length $m$. Then, by emphasizing the entries of $\mathcal{P}$ corresponding to $\Gamma_1$, $\Gamma_2$, we construct matrices $B_{\Gamma_1}(\epsilon), B_{\Gamma_2}(\epsilon') \in \mathcal{Q}(\mathcal{P})$ as in \ref{e2}, for some $\epsilon, \epsilon'>0$, such that $E_m(B_{\Gamma_1}(\epsilon))=-E_m(B_{\Gamma_2}(\epsilon'))$.
		
		Since $E_m(B)$ is equal to the product of all the nonzero eigenvalues of $B$, for any $B\in \mathcal{Q}(\mathcal{P})$ and the nonreal eigenvalues of $B$ occur in conjugate pairs, $E_m(B_{\Gamma_1}(\epsilon))=-E_m(B_{\Gamma_2}(\epsilon'))$ implies that $i_-(B_{\Gamma_1}(\epsilon))\neq i_-(B_{\Gamma_2}(\epsilon'))$. Therefore, $\mathcal{P}$ does not require a unique inertia.
	\end{proof}
	
	The converse of Theorem \ref{xnl1} is not true in general, as the following example shows.
	\begin{eg}\label{xxeg2.2}
		Suppose that $\mathcal{P}$ is an irreducible combinatorially symmetric sign pattern with a $0$-diagonal, whose underlying signed directed graph $D$ is given in Fig.\ref{fignx1}.
		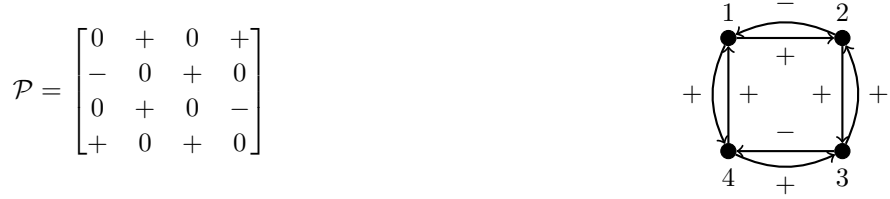
\begin{figure}[H]
			\centering
			\begin{minipage}{.45\textwidth}
				\vspace{-1cm}
				\[
				\mathcal{P}=\begin{bmatrix}
					0 & + & 0 & + \\
					- & 0 & + & 0 \\
					0 & + & 0 & - \\
					+ & 0 & + & 0
				\end{bmatrix}
				\]
			\end{minipage}
			\hspace{0.05\textwidth}
			\begin{minipage}{.45\textwidth}
				\centering
				\tikzset{node distance=1cm}
				\begin{tikzpicture}[
					vertex/.style={draw, circle, inner sep=0pt, minimum size=2mm, fill=black},
					edge/.style={thick, ->}
					]
					\node[vertex,label=above:1] (v1) at (-0.75, 0.75) {};
					\node[vertex,label=above:2] (v2) at ( 0.75, 0.75) {};
					\node[vertex,label=below:3] (v3) at ( 0.75,-0.75) {};
					\node[vertex,label=below:4] (v4) at (-0.75,-0.75) {};
					
					\draw[edge] (v1) -- node[below] {$+$} (v2);
					\draw[edge, bend right=25] (v2) to node[above] {$-$} (v1);
					
					\draw[edge] (v2) -- node[left] {$+$} (v3);
					\draw[edge, bend right=25] (v3) to node[right] {$+$} (v2);
					
					\draw[edge] (v3) -- node[above] {$-$} (v4);
					\draw[edge, bend right=25] (v4) to node[below] {$+$} (v3);
					
					\draw[edge] (v4) -- node[right] {$+$} (v1);
					\draw[edge, bend right=25] (v1) to node[left] {$+$} (v4);
				\end{tikzpicture}
				\caption{Signed directed graph \(D\) of \(\mathcal{P}\).}
				\label{fignx1}
			\end{minipage}
			
		\end{figure}
		
		Then the maximum length of a composite cycle in $D$ is $4$, and all such cycles have the same sign, which is positive. However, consider	$$\begin{bmatrix}
			0&a&0&b\\ -c&0&d&0\\ 0&e&0&-f\\ g&0&h&0
		\end{bmatrix}\in \mathcal{Q}(\mathcal{P}), ~\text{where}~ a,b,c,d,e,f,g,h>0.$$ \begin{itemize}
			\item[i.] Let $B_1 \in \mathcal{Q}(\mathcal{P})$ with $a = b = c = d = e = f = g = h = 1$. Then $\sigma(B_1) = \{1+i, 1-i,-1-i,-1+i\}$, and $\In(B_1) = (2, 2, 0)$.
			\item[ii.] Let $B_2 \in \mathcal{Q}(\mathcal{P})$ with $a = 11$, $ b = c = d = e = f = g = h = 1$. Then $\sigma(B_2) = \{2i,-2i,\sqrt{6}i,-\sqrt{6}i\}$, and $\In(B_2) = (0, 0, 4)$.
		\end{itemize}
		Therefore, $\mathcal{P}$ does not require a unique inertia.
	\end{eg}
	
	\begin{theorem} \label{l3.1n}
		If $\mathcal{P}\in \mathcal{Q}_n$ is an irreducible tridiagonal sign pattern with a $0$-diagonal such that $\mathcal{P}$ allows repeated zero eigenvalues, then $\mathcal{P}$ does not require a unique inertia.
	\end{theorem}
	\begin{proof} Since $\mathcal{P}$ allows zero eigenvalues, $n$ must be odd. Therefore, the maximum length of a composite cycle in the signed directed graph $D$ of $\mathcal{P}$ is $n-1$, which is even. Since $\mathcal{P}$ allows repeated zero eigenvalues, the underlying signed directed graph $D$ of $\mathcal{P}$ contains at least two composite cycles of length $n - 1$ with opposite signs.
		Therefore, from Lemma~\ref{xnl1}, it follows that $\mathcal{P}$ does not require a unique inertia.
	\end{proof}
	
	The converse of Theorem~\ref{l3.1n} is not true in general, as the following example shows.
	\begin{eg} Let $$\mathcal{P}=\begin{bmatrix}
			0&-&0&0\\ +&0&+&0\\ 0&+&0&-\\ 0&0&+&0
		\end{bmatrix}$$ be an irreducible tridiagonal sign pattern with a $0$-diagonal. Any real matrix $B\in \mathcal{Q}(\mathcal{P})$ is similar to	$$\begin{bmatrix}
			0&-1&0&0\\ a&0&1&0\\ 0&b&0&-1\\ 0&0&c&0
		\end{bmatrix}~\text{where} ~a,b,c>0.$$ Since $\det(B)\neq 0$, $\mathcal{P}$ does not allow any zero eigenvalue. However,
		\begin{itemize}
			\item[i.] let $B_1 \in \mathcal{Q}(\mathcal{P})$ with $a=c=1$, $b=4$. Then $\sigma(B_1)=\{1, 1, -1,-1\}$, and hence $In(B_1)=(2,2,0)$,
			\item[ii.] let $B_2 \in \mathcal{Q}(\mathcal{P})$ with $a=8$, $b=c=2$. Then $\sigma(B_2)=\{2i,2i, -2i,  -2i\}$, and hence $\In(B_2) = (0, 0, 4)$.
		\end{itemize} Hence, $\mathcal{P}$ does not requires a unique inertia.
	\end{eg}

	The conclusion of Theorem \ref{l3.1n} is not necessarily true for tree sign patterns of odd order with a $0$-diagonal.
	\begin{eg}
		Let $$\mathcal{P}=\begin{bmatrix}
			0&+&+&+&+\\+&0&0&0&0\\+&0&0&0&0\\+&0&0&0&0\\+&0&0&0&0
		\end{bmatrix}.$$ Then $rank(B)=2$ for all $B\in \mathcal{Q}(\mathcal{P})$, so $\mathcal{P}$ allows repeated zero eigenvalues. However, $\mathcal{P}$ is a symmetric tree sign pattern with a $0$-diagonal, and by Remark~4.1\cite{2018}, every matrix $B\in\mathcal{Q}(\mathcal{P})$ is similar to a symmetric matrix, so by Theorem 4.3\cite{2001}, $\mathcal{P}$ requires a unique inertia.
	\end{eg}

	\begin{definition}
		Let $\mathcal{P}\in \mathcal{Q}_n$ be a tree sign pattern with a $0$-diagonal, with the underlying signed undirected graph $G$. Then $\mathcal{P}_-$ is defined to be the sign pattern whose underlying signed undirected graph $G_-$ is obtained from $G$ by taking the opposite sign of each signed edge of $G$.
	\end{definition}
	
	\begin{eg} Let $\mathcal{P}\in \mathcal{Q}_n$ be a tree sign pattern with a $0$-diagonal, whose underlying signed undirected graph is $G$  given in Fig.\ref{figx1}.
		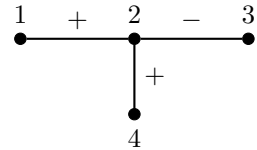
\begin{figure}[H]
			\centering
			\begin{minipage}{.45\textwidth}
				\vspace{-1cm}
				\[
				\mathcal{P}=\begin{bmatrix}
					0 & + & 0 & 0 \\
					+ & 0 & + & + \\
					0 & - & 0 & 0 \\
					0 & + & 0 & 0
				\end{bmatrix}
				\]
			\end{minipage}
			\hspace{0.05\textwidth}
			\begin{minipage}{.45\textwidth}
				\centering
				\tikzset{node distance=1cm}
				\begin{tikzpicture}[
					vertex/.style={draw, circle, inner sep=0pt, minimum size=.15cm, fill=black},
					edge/.style={thick}
					]
					\node[vertex,label=above:1] (v1) at (-1.5,0) {};
					\node[vertex,label=above:2] (v2) at (0,0) {};
					\node[vertex,label=above:3] (v3) at (1.5,0) {};
					\node[vertex,label=below:4] (v4) at (0,-1) {};
					
					\draw[edge] (v1) -- node[above] {$+$} (v2);
					\draw[edge] (v2) -- node[above] {$-$} (v3);
					\draw[edge] (v4) -- node[right] {$+$} (v2);
				\end{tikzpicture}
				\caption{Signed undirected graph $ G$ of $\mathcal{P}$.}
				\label{figx1}
			\end{minipage}
		\end{figure}

		Then the underlying graph corresponding to $\mathcal{P}_-$ is $G_-$ given in Fig.\ref{fign2}.
		\begin{figure}[H]
			\centering
			\begin{minipage}{0.45\textwidth}
				\vspace{-1cm}
				\[
				\mathcal{P}_-=
				\begin{bmatrix}
					0 & + & 0 & 0 \\
					- & 0 & + & + \\
					0 & + & 0 & 0 \\
					0 & - & 0 & 0
				\end{bmatrix}
				\]
			\end{minipage}\hfill
			\begin{minipage}{0.45\textwidth}
				\centering
				\tikzset{
					vertex/.style={draw, circle, inner sep=0pt, minimum size=.15cm, fill=black},
					edge/.style={thick}
				}
				\begin{tikzpicture}
					\node[vertex,label=above:1](v1)at(-1.5,0){};
					\node[vertex,label=above:2](v2)at(0,0){};
					\node[vertex,label=above:3](v3)at(1.5,0){};
					\node[vertex,label=below:4](v4)at(0,-1){};
					
					\draw[edge](v1)--node[above]{$-$}(v2);
					\draw[edge](v2)--node[above]{$+$}(v3);
					\draw[edge](v4)--node[right]{$-$}(v2);
				\end{tikzpicture}
				\caption{Signed undirected graph $G_-$ of $\mathcal{P}_-$.}
				\label{fign2}
			\end{minipage}
		\end{figure}
		
	\end{eg}
	\begin{lemma} \label{3.1} Let $\mathcal{P}\in \mathcal{Q}_n$ be a tree sign pattern with a $0$-diagonal.
		If $\lambda$ is an eigenvalue of $B=[b_{ij}]\in \mathcal{Q}(\mathcal{P})$ with algebraic multiplicity $m$, then $i\lambda$ is an eigenvalue of $B_-=[b^-_{ij}] \in \mathcal{Q}(\mathcal{P}_-)$ with algebraic multiplicity $m$, where $|b_{ij}|=|b^-_{ij}|$ for all $i,j$.
	\end{lemma}
	\begin{proof}
		Suppose that $t$ is the maximum length of a composite cycle in the underlying signed directed graph $D$ of $\mathcal{P}$. Then the characteristic polynomial of $B$ is $$\Ch_B(x)=x^n-E_1(B)x^{n-1}+E_2(B)x^{n-2}-\cdots+(-1)^tE_t(B)x^{n-t},$$ where $E_k(B)$ for $1\leq k\leq t$, is the sum of all the cycles (simple or composite) of length $k$ in $B$, properly signed. 
		
		Let $n$ be even. Since $\mathcal{P}$ is a tree sign pattern of even order, $E_k(B)=0$ for all odd $k$. Therefore, $$\Ch_B(x)=x^n+\sum_{k=1}^{\frac{t}{2}}E_{2k}(B)x^{n-2k}.$$
		Since $\mathcal{P}_-$ is obtained from $\mathcal{P}$ by taking the opposite sign of each edge of the underlying signed undirected graph of $\mathcal{P}$, and $|b_{ij}|=|b_{ij}^-|$, we have 
		$$\Ch_{B_-}(x)=x^n+\sum_{k=1}^{\frac{t}{2}}(-1)^kE_{2k}(B)x^{n-2k}.$$
		Therefore, if $\lambda$ is a root of $\Ch_B(x)$ with multiplicity $m$, then $i\lambda$ is a root of $\Ch_{B_-}(x)$ with multiplicity $m$.
		
		If $n$ is odd, then $\Ch_B(x) = x(x^{n-1}+\sum_{k=1}^{\frac{t}{2}}E_{2k}(B)x^{n-2k-1})$. So, the result follows from the previous part.
	\end{proof}

	The following result is by Lin et al. \cite{2018}.

	\begin{theorem} [Theorem 2.4, \cite{2018}]\label{th4}
		Let $\mathcal{P}$ be an $n \times n$ sign pattern. Then $\mathcal{P}$ requires a unique inertia if and only if $i_0(A)$ is a fixed number for all $A \in \mathcal{Q}(\mathcal{P})$. Moreover, $\mathcal{P}$ requires a unique refined inertia if and only if $i_z(A)$ and $i_p(A)$ are two fixed numbers for all $A \in \mathcal{Q}(\mathcal{P})$.
	\end{theorem}
	
	The following is an immediate consequence of the previous theorem.
	\begin{theorem}\label{lem2}
		Let $\mathcal{P}\in \mathcal{Q}_n$ and $A,B\in \mathcal{Q}(\mathcal{P})$ be such that $i_0(A)+(n-i_0(B))\geq n+1$, then $\mathcal{P}$ does not require a unique inertia. 
	\end{theorem}
	\begin{proof}
		Since $i_0(A)+(n-i_0(B))\geq n+1$, $i_0(A)\neq i_0(B)$, therefore, by Theorem \ref{th4}, $\mathcal{P}$ does not require a unique inertia.
	\end{proof}
	
	
	\begin{theorem} \label{th2.3}
		Let $\mathcal{P}\in \mathcal{Q}_n$ be a tree sign pattern with a $0$-diagonal. Then $\mathcal{P}$ requires a unique inertia if and only if $\mathcal{P}_-$ is consistent.
	\end{theorem}
	\begin{proof}
		Suppose that $\mathcal{P}$ requires a unique inertia. By Theorem \ref{th4}, every $A=[a_{ij}]\in \mathcal{Q}(\mathcal{P})$ has a fixed number of eigenvalues with zero real part.  Since $iA$ and $A_-$ have the same characteristic polynomial, $A_-$ has a fixed number of real eigenvalues, where $A_-=[a^-_{ij}]\in \mathcal{Q}(\mathcal{P_-})$ with $|a_{ij}|=|a^-_{ij}|$ for all $i,j$. Since $A\in \mathcal{Q}(\mathcal{P})$ and $A_- \in \mathcal{Q}(\mathcal{P_-})$ are arbitrary, $\mathcal{P}_-$ is consistent.
		
		The converse follows similarly.
	\end{proof}
	
		
	\begin{cor}\label{c2.2}
		Let $\mathcal{P}\in \mathcal{Q}_n$ be a tree sign pattern with a $0$-diagonal. Then $\mathcal{P}_-$ requires a unique inertia if and only if $\mathcal{P}$ is consistent.
	\end{cor}

	The following result, by Eschenbach et al. \cite{1}, gives a necessary condition for irreducible tridiagonal sign patterns with a $0$-diagonal to be consistent. We give a similar necessary condition for such sign patterns to require a unique inertia.

	\begin{theorem} [Theorem 2.8, \cite{1}] \label{th11}
		Let $\mathcal{A}$ be an irreducible tridiagonal pattern with a $0$-diagonal. 
		Then $\mathcal{A}$ is consistent only if the signed undirected graph of $\mathcal{A}$ has, at most, one 
		maximal signed path with odd length. 
	\end{theorem}
	
	The next result shows that a path sign pattern with a $0$-diagonal requiring a unique inertia, also has at most one maximal signed path of odd length.

	\begin{theorem}\label{th2.51}
		Let $\mathcal{P}\in \mathcal{Q}_n$ be an irreducible tridiagonal sign pattern with a $0$-diagonal. Then $\mathcal{P}$ requires a unique inertia only if the signed undirected graph of $\mathcal{P}$ has, at most, one maximal signed path with odd length.
	\end{theorem}
	\begin{proof}
		By Theorem \ref{th2.3}, $\mathcal{P}$ requires a unique inertia if and only if $\mathcal{P}_-$ is consistent. Also, by Theorem \ref{th11}, if $\mathcal{P}_-$ is consistent, then the signed undirected graph of $\mathcal{P}_-$ has at most one maximal signed path with odd length.
		
		Since the underlying signed undirected graph of $\mathcal{P}$ is obtained from the underlying signed undirected graph of $\mathcal{P}_-$ by taking the opposite sign of each signed edge, $\mathcal{P}$ has at most one maximal signed path of odd length. 
	\end{proof}

	However, the above condition is not sufficient for a tridiagonal sign pattern with a $0$-diagonal to require a unique inertia. 
	\begin{eg} \label{neg11}
		Let       
		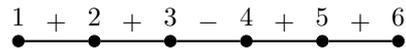
\begin{figure}[H]
			\centering
			\begin{minipage}{.45\textwidth}
				\vspace{-1cm}
				\[
				\mathcal{P}=\begin{bmatrix}
					0 & + & 0 & 0 & 0 & 0 \\
					+ & 0 & + & 0 & 0 & 0 \\
					0 & + & 0 & - & 0 & 0 \\
					0 & 0 & + & 0 & + & 0 \\
					0 & 0 & 0 & + & 0 & + \\
					0 & 0 & 0 & 0 & + & 0
				\end{bmatrix}
				\]
			\end{minipage}
			\hspace{0.05\textwidth}
			\begin{minipage}{.45\textwidth}
				\centering
				\tikzset{node distance=1cm}
				\begin{tikzpicture}[
					vertex/.style={draw, circle, inner sep=0pt, minimum size=.15cm, fill=black},
					edge/.style={thick}
					]
					\node[vertex,label=above:1] (v1) at (-2,0) {};
					\node[vertex,label=above:2] (v2) at (-1,0) {};
					\node[vertex,label=above:3] (v3) at ( 0,0) {};
					\node[vertex,label=above:4] (v4) at ( 1,0) {};
					\node[vertex,label=above:5] (v5) at ( 2,0) {};
					\node[vertex,label=above:6] (v6) at ( 3,0) {};
					
					\draw[edge] (v1) -- node[above] {$+$} (v2);
					\draw[edge] (v2) -- node[above] {$+$} (v3);
					\draw[edge] (v3) -- node[above] {$-$} (v4);
					\draw[edge] (v4) -- node[above] {$+$} (v5);
					\draw[edge] (v5) -- node[above] {$+$} (v6);
				\end{tikzpicture}
				\caption{Signed undirected graph $ G$ of $\mathcal{P}$.}
				\label{xnfig11}
			\end{minipage}
		\end{figure}
	be an irreducible tridiagonal sign pattern with a $0$-diagonal whose underlying signed undirected graph is $G$ given in Fig~\ref{xnfig11}. Clearly, $G$ has exactly one maximal signed path with odd length. Any real matrix $B\in \mathcal{Q}(\mathcal{P}),$ is similar to
	
	$$\begin{bmatrix}
		0&1&0&0&0&0\\ a&0&1&0&0&0\\ 0&b&0&-1&0&0\\ 0&0&c&0&1&0\\0&0&0&d&0&1\\0&0&0&0&e&0
	\end{bmatrix}, ~\text{where}~ a,b,c,d,e>0.$$
	\begin{itemize}
		\item[i.] If $a=e=\frac{1}{20}$, $b=d=5$, $c=20$ then $\sigma(B)=\{ 2.4401i, - 2.4401i, 1.9859i,- 1.9859i, 0.0461i, - 0.0461i\}$, so $\In(B)=(0,0,6)$.
		\item[ii.] If $a=b=e=d=e=1$ then $\sigma(B)=\{-1.3071 + 0.2151i, -1.3071 - 0.2151i, 1.3071 + 0.2151i, 1.3071 - 0.2151i, 0.5698i, - 0.5698i\}$, so $\In(B)=(2,2,2)$.
	\end{itemize}
	Therefore, $\mathcal{P}$ does not require a unique inertia, although it has exactly one maximal signed path with odd length.
\end{eg}

However, for irreducible tree sign patterns with a $0$-diagonal, the conclusion of Theorem~\ref{th2.51} may not hold.

\begin{eg}
	Let $\mathcal{P}$ be an irreducible tree sign pattern with a $0$-diagonal, whose underlying graph $G$ is given in Fig.~\ref{X.1}.
	\begin{figure}[H]
		\centering
		\begin{minipage}{.45\textwidth}
			\vspace{-1cm}
			\[
			\mathcal{P}=\begin{bmatrix}
				0 & - & 0 & 0 & 0 & 0 \\
				+ & 0 & + & 0 & 0 & 0 \\
				0 & + & 0 & + & 0 & + \\
				0 & 0 & + & 0 & - & 0 \\
				0 & 0 & 0 & + & 0 & 0 \\
				0 & 0 & + & 0 & 0 & 0
			\end{bmatrix}
			\]
		\end{minipage}
		\hspace{0.05\textwidth}
		\begin{minipage}{.45\textwidth}
			\centering
			\tikzset{node distance=1cm}
			\begin{tikzpicture}[
				vertex/.style={draw, circle, inner sep=0pt, minimum size=0.15cm, fill=black},
				edge/.style={thick}
				]
				\node[vertex, label=above:$1$] (v1) at (0,0) {};
				\node[vertex, label=above:$2$] (v2) at (1.25,0) {};
				\node[vertex, label=above:$3$] (v3) at (2.5,0) {};
				\node[vertex, label=above:$4$] (v4) at (3.75,0) {};
				\node[vertex, label=above:$5$] (v5) at (5,0) {};
				\node[vertex, label=below:$6$] (v6) at (2.5,-1.25) {};
				
				\draw[edge] (v1) -- node[above] {$-$} (v2);
				\draw[edge] (v2) -- node[above] {$+$} (v3);
				\draw[edge] (v3) -- node[above] {$+$} (v4);
				\draw[edge] (v4) -- node[above] {$-$} (v5);
				\draw[edge] (v3) -- node[right] {$+$} (v6);
			\end{tikzpicture}
			\caption{Signed undirected graph $G$ of $\mathcal{P}$.}
			\label{X.1}
		\end{minipage}
	\end{figure}
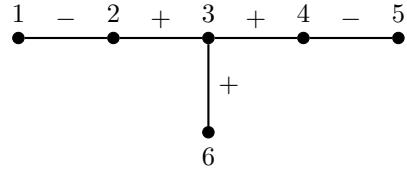
	
	Then any real matrix $B\in \mathcal{Q}(\mathcal{P})$ is similar to\\		
	$$
	\begin{bmatrix}
		0 & -a &  0&0  &0  & 0  \\
		1 & 0 & b  & 0&0 &0    \\
		0& 1 & 0  & c  & 0&e  \\
		0&0 & 1   & 0  & -d &0  \\
		0&0  &  0   & 1 & 0 & 0\\
		0&0  &  1 &  0& 0  & 0
	\end{bmatrix}\in \mathcal{Q}(\mathcal{P}),~ where~ a,b,c,d,e>0.$$
	The characteristic polynomial of $B$ is $\Ch_{B}(x)= x^6+E_2 x^4+E_4x^2+E_6$, where $E_2=a-b-c+d-e, E_4= -ac+ad-bd-ea-eb, E_6=-ead$.
	
	\textbf{Case 1:} Suppose $E_2 \geq 0$.  
	For all $x \in \mathbb{R} \setminus \{0\}$, we have
	\[
	\sgn(\Ch_{B}(x)) = 
	\begin{cases} 
		(+)+(+/0)(+)+(+/0)(+)+(-) & \text{if } E_4 \geq 0 \\[6pt]
		(+)+(+/0)(+)+(-)(+)+(-) & \text{if } E_4 < 0.
	\end{cases}
	\]
	Thus $V_+(x) = 1$ and $V_-(x) = 1$. By Descartes' rule of signs, $\Ch_B(x)$ has exactly one positive and one negative real root. Since $\det(B) \neq 0$, the eigenvalue frequency is $S_B = (2,4)$.
	
	\medskip
	
	\textbf{Case 2:} Suppose $E_2 < 0$. Then
	\begin{equation} \label{eq1}
		a - b - c + d - e < 0.
	\end{equation}
	If $E_4 \geq 0$, then
	\begin{equation} \label{eq2}
		-ac + ad - bd - ea - eb \geq 0
	\end{equation}
	\[
	\implies d \geq c+e+\left(\frac{b}{a}\right)d+\left(\frac{e}{a}\right)b 
	\;\;\implies\;\; d > c+e.
	\]
	Now, from \eqref{eq1},
	\[
	a+(d-c-e)-b < 0 \;\;\implies\;\; a < b.
	\]
	On the other hand, from \eqref{eq2},
	\[
	(a-b)d \geq ac+ea+ed > 0 \;\;\implies\;\; a > b,
	\]
	a contradiction. Therefore $E_4 < 0$.
	Hence, for all $x \in \mathbb{R} \setminus \{0\}$,
	\[
	\sgn(\Ch_{B}(x)) = (+)+(-)(+)+(-)(+)+(-).
	\]
	Thus $V_+(x) = 1$ and $V_-(x) = 1$. By Descartes' rule of signs, $\Ch_B(x)$ has exactly one positive and one negative real root. Since $\det(B) \neq 0$, the eigenvalue frequency is again $S_B = (2,4)$.
	In both cases, $S_B = (2,4)$ for all $B \in \mathcal{Q}(\mathcal{P})$, hence
	$
	S_{\mathcal{P}} = (2,4).
	$
	
	Therefore, by Corollary~\ref{c2.2}, $\mathcal{P}_-$ requires a unique inertia. However, the underlying signed undirected graph of $\mathcal{P}_-$ contains more than one maximal signed path of odd length.
\end{eg}

Next, we focus on identifying certain patterns that cannot occur as subpatterns of tridiagonal sign patterns requiring a unique inertia.

Let, $$\mathcal{P}_4=\begin{bmatrix}
	0&+&0&0\\+&0&-&0\\0&+&0&+\\0&0&+&0
\end{bmatrix}~\text{and}~\mathcal{P}_6=\begin{bmatrix}
	0&+&0&0&0&0\\+&0&+&0&0&0\\0&+&0&-&0&0\\0&0&+&0&+&0\\0&0&0&+&0&+\\0&0&0&0&+&0
\end{bmatrix}.$$
\begin{theorem}\label{nth3.4}
	Let $\mathcal{P}\in \mathcal{Q}_n$ be an irreducible tridiagonal sign pattern with a $0$-diagonal. If $\mathcal{P}_4$ or ${\mathcal{P}_4}_-$ is a submatrix of $\mathcal{P}$, then $\mathcal{P}$ does not require a unique inertia.
\end{theorem}

In order to prove the above theorem, we first show that each of $\mathcal{P}_4$, ${\mathcal{P}_4}_-$, $\mathcal{P}_6$, ${\mathcal{P}_6}_-$ does not require a unique inertia. Also, we show that if $\mathcal{P}$ is one of the above, then there exist real matrices $B_1, B_2\in \mathcal{Q}(\mathcal{P})$ such that $\In(B_1)\neq \In(B_2)$, where both $B_1, B_2$ have all distinct eigenvalues. 

Since each of $\mathcal{P}_4$, ${\mathcal{P}_4}_-$ has more than one maximal signed path with odd length, by Theorem \ref{th2.51}, $\mathcal{P}_4$, ${\mathcal{P}_4}_-$ does not require a unique inertia. The following example shows that there exist $B_1, B_2\in \mathcal{Q}(\mathcal{P}_4)$ such that $\In(B_1)\neq \In(B_2)$ and both $B_1, B_2$ have all distinct eigenvalues.
\begin{eg} \label{eg3.1}
	Consider $B\in \mathcal{Q}(\mathcal{P}_4),$ then $B$ is similar to
	
	$$\begin{bmatrix}
		0&1&0&0\\ a&0&-1&0\\ 0&b&0&1\\ 0&0&c&0
	\end{bmatrix},~\text{where}~ a,b,c>0.$$ Then, $\Ch_B(x)=x^4+(-a+b-c)x^2+ac$. 
	\begin{itemize}
		\item[i.] If $a=b=c=1$ then $\sigma(B_1)=\{\frac{\sqrt{3}}{2}+\frac{i}{2},\frac{\sqrt{3}}{2}- \frac{i}{2},-\frac{\sqrt{3}}{2}+ \frac{i}{2},- \frac{\sqrt{3}}{2}-\frac{i}{2}\}$, $\In(B_1)=(2,2,0)$.
		\item[ii.] If $a=1$, $b=10$ and $c=4$ then $\sigma(B_2)=\{i,-i, 2i,-2i\}$, $\In(B_2)=(0,0,4)$.
	\end{itemize}
	By Lemma \ref{3.1}, there also exists $B_1,B_2\in \mathcal{Q}({\mathcal{P}_{4}}_-)$ such that both $B_1, B_2$ have all distinct eigenvalues and $\In(B_1)\neq \In(B_2)$.
\end{eg}

\begin{eg} \label{eg2.21}
	From Example \ref{neg11}, there exists $B_1, B_2 \in \mathcal{Q}(\mathcal{P}_6)$ such that both $B_1, B_2$ have all distinct eigenvalues and
	$\In(B_1)\neq \In(B_2)$. Similarly, Lemma \ref{3.1} implies that there exists $B_1, B_2 \in \mathcal{Q}({\mathcal{P}_{6}}_-)$ such that both $B_1, B_2$ have all distinct eigenvalues and $\In(B_1)\neq \In(B_2)$.
\end{eg}

\begin{proof}[Proof of Theorem \ref{nth3.4}]
	Suppose that $\mathcal{P}=[p_{ij}]\in \mathcal{Q}_n$ is an irreducible tridiagonal sign pattern with a $0$-diagonal such that $\mathcal{P}_4$ is a submatrix of $\mathcal{P}$, then $n\geq 4$. Let $G$ be the signed undirected graph corresponding to $\mathcal{P}$. If the number of maximal signed paths of odd length in $G$ is more than one, then by Theorem \ref{th2.51}, $\mathcal{P}$ does not require a unique inertia, so we assume that the number of maximal odd length paths in $G$ is exactly one, which implies that $n$ is even.
	
	
	For $n=4$, by Example \ref{eg3.1}, $\mathcal{P}$ does not require a unique inertia. For $n=6$, since the number of maximal signed paths of odd length in $G$ is at most one, $\mathcal{P}$ is equivalent to $\mathcal{P}_6$, so by Example \ref{eg2.21}, $\mathcal{P}$ does not require a unique inertia. 
	
	For $n=2k$ with $k\geq4$, the underlying signed directed graph of $\mathcal{P}$ is given by $D = \alpha_1\alpha_2\cdots \alpha_{n-1}$, where each $\alpha_i = p_{i,i+1}p_{i+1,i}$ for $i = 1, 2, \dots, n-1$ represents a $2$-cycle in $D$. Let $D_4 = \alpha_l\alpha_{l+1}\alpha_{l+2}$ be the subgraph of $D$, which is the signed directed graph corresponding to $\mathcal{P}_4$. 
	Since the number of maximal signed paths of odd length in $G$ is one, $D_6=\alpha_{l-1}\alpha_{l}\alpha_{l+1}\alpha_{l+2}\alpha_{l+3}$ is also a subgraph of $D$, which is the signed directed graph corresponding to $\mathcal{P}_6$ and
	$$\gamma=\alpha_1\alpha_3\cdots\alpha_{l-3}\alpha_{l+5}\cdots\alpha_{n-1},$$ is a composite cycle in $D\setminus V(D_6)$ of length $n-6$.

	From Example \ref{eg2.21}, there exist two real matrices $B_1=[(b_1)_{ij}]$, $B_2=[(b_2)_{ij}]$ in $\mathcal{Q}(\mathcal{P}_6)$ with $\In(B_1)\neq \In(B_2)$ such that both $B_1, B_2$ have all distinct eigenvalues. Let,
	$$M>\max\{1,|\lambda|: \lambda~ \text{is}~ \text{an}~ \text{eigenvalue}~ \text{of}~\text{either} ~B_1~\text{or}~B_2\}$$ and define the real matrices $B_{r,\gamma}(0)=[b_{r,\gamma}(0)_{ij}]$ for $r=1,2$, as
	$$(b_{r,\gamma}(0))_{ij}= \begin{cases}
		M^p & \text{if}\ p_{ij}=+~\text{and}~ \text{is}~ \text{in}~ \alpha_p,~\text{for}~p=1,3,...,l-3,l+5,...,n-1 \\
		-M^p & \text{if}\ p_{ij}=-~\text{and}~ \text{is}~ \text{in}~ \alpha_p,~\text{for}~p=1,3,...,l-3,l+5,...,n-1 \\
		
		(b_{r})_{i+2-l,j+2-l} & \text{if}\ l-1\leq i,j\leq l+4 \\
		
		0 & \text{elsewhere.}
	\end{cases}$$
	Then the eigenvalues of $B_{1,\gamma}(0)$, $B_{2,\gamma}(0)$ are algebraically simple and the eigenvalues of $B_{r,\gamma}(0)$, $r=1,2$ are the eigenvalues of $B_r$ together with the second complex roots of $M^{2p}$ or $-M^{2p}$, depending on the sign of $\alpha_p$, for $p=1,3,...,l-3,l+5,...,n-1$. Note that both $B_{1,\gamma}(0)$, $B_{2,\gamma}(0)$ have all distinct eigenvalues and $\In(B_{1,\gamma}(0))\neq \In(B_{2,\gamma}(0))$. For $\epsilon>0$ define the perturbed matrices $B_{r,\gamma}(\epsilon)=[b_{r,\gamma}(\epsilon)_{ij}]$ for $r=1,2$, as
	\[
	(b_{r,\gamma}(\epsilon))_{ij} = 
	\begin{cases}
		(b_{r,\gamma}(0))_{ij} & \text{if } p_{ij} \neq 0 \text{ and } (b_{r,\gamma}(0))_{ij} \neq 0 \\
		\epsilon & \text{if } p_{ij} = + \text{ and } (b_{r,\gamma}(0))_{ij} = 0 \\
		-\epsilon & \text{if } p_{ij} = - \text{ and } (b_{r,\gamma}(0))_{ij} = 0 \\
		0 & \text{otherwise.}
	\end{cases}
	\]
	Then $B_{1,\gamma}(\epsilon),B_{2,\gamma}({\epsilon})\in \mathcal{Q}(\mathcal{P})$. For $\epsilon>0$ sufficiently small, all the eigenvalues of $B_{1,\gamma}(\epsilon)$ remain close to the eigenvalues of $B_{1,\gamma}(0)$.   The eigenvalues of $B_{1,\gamma}(\epsilon)$ close to the real eigenvalues of $B_{1,\gamma}(0)$ of the form $\lambda$, $-\lambda$ must be real and of the same form, since they cannot merge to form conjugate pairs of nonreal eigenvalues. Similarly, the eigenvalues of $B_{1,\gamma}(\epsilon)$ close to the nonreal eigenvalues of $B_{1,\gamma}(0)$ of the form $i\lambda$, $-i\lambda$ must be nonreal and of the same form. Therefore, for $\epsilon>0$, sufficiently small, $\In(B_{1,\gamma}(0))=\In(B_{1,\gamma}(\epsilon))$.
	
	Similarly, since $B_{2,\gamma}(0)$ has all distinct eigenvalues, there exists $\bar{\epsilon}>0$ sufficiently small such that $\In(B_{2,\gamma}(0))=\In(B_{2,\gamma}(\bar{\epsilon}))$. Since $\In(B_{1,\gamma}(0))\neq \In(B_{2,\gamma}(0))$, $\In(B_{1,\gamma}(\epsilon))\neq \In(B_{2, \gamma}(\bar{\epsilon}))$ therefore, $\mathcal{P}$ does not require a unique inertia.
	
	Similarly, if $\mathcal{P}_{4-}$ is a submatrix of $\mathcal{P}$, then it can be shown that $\mathcal{P}$ does not require a unique inertia.
\end{proof}

When the order of $\mathcal{P}$ is greater than or equal to $6$, we specify another pattern which also cannot be a subpattern of a tridiagonal sign pattern requiring a unique inertia.

Let, $$\mathcal{P}_6'=\begin{bmatrix}
	0&+&0&0&0&0\\+&0&-&0&0&0\\0&+&0&-&0&0\\0&0&+&0&-&0\\0&0&0&+&0&+\\0&0&0&0&+&0
\end{bmatrix}, ~~~
\mathcal{P}_8'=\begin{bmatrix}
	0&+&0&0&0&0&0&0\\+&0&+&0&0&0&0&0\\0&+&0&-&0&0&0&0\\0&0&+&0&-&0&0&0\\0&0&0&+&0&-&0&0\\0&0&0&0&+&0&+&0\\0&0&0&0&0&+&0&+\\0&0&0&0&0&0&+&0
\end{bmatrix}.$$
The following examples show that there exists $B_1, B_2\in \mathcal{Q}(\mathcal{P}_6')$ (or, $\mathcal{Q}(\mathcal{P}_8'$)) such that both $B_1,B_2$ have all distinct eigenvalues and $\In(B_1)\neq \In(B_2)$, which implies, $\mathcal{P}_6'$ (or, $\mathcal{P}_8'$) does not require a unique inertia.

\begin{eg} \label{ega21} 
	Take $B\in \mathcal{Q}(\mathcal{P}_6'),$ then $B$ is similar to
	
	$$\begin{bmatrix}
		0&1&0&0&0&0\\a&0&-1&0&0&0\\0&b&0&-1&0&0\\0&0&c&0&-1&0\\0&0&0&d&0&1\\0&0&0&0&e&0
	\end{bmatrix}, ~\text{where}~ a,b,c,d,e>0.$$ \\
	
	\begin{itemize}
		\item[i.] If $B_1\in \mathcal{Q}(\mathcal{P}_6')$ is such that $a=10$, $b=c=d=e=1$, then $\sigma(B_1)=\{-3.0148,  3.0148, 0.7983, -0.7983, 1.3139i, -1, 3139i\}$, so $In(B_1)=(2,2,2)$ and $B_1$ has all distinct eigenvalues. 
		\item[ii.] If $B_2\in \mathcal{Q}(\mathcal{P}_6')$ is such that $a=e=\frac{1}{20}$, $b=d=5$, $c=20$, then $\sigma(B_2)=\{ 5.3970i,-5.3970i, 0.8776i, -0.8776i, 0.0472i, 0.0472i\}$, so $In(B_2)=(0,0,6)$ and $B_2$ has all distinct eigenvalues.
	\end{itemize}
\end{eg}

\begin{eg}\label{ega11}  Take $B\in \mathcal{Q}(\mathcal{P}_8'),$ then $B$ is similar to
	$$\begin{bmatrix}
		0&1&0&0&0&0&0&0\\a&0&1&0&0&0&0&0\\0&b&0&-1&0&0&0&0\\0&0&c&0&-1&0&0&0\\0&0&0&d&0&-1&0&0\\0&0&0&0&e&0&1&0\\0&0&0&0&0&f&0&1\\0&0&0&0&0&0&g&0
	\end{bmatrix},~\text{where}~ a,b,c,d,e,f,g,h>0.$$ 
	
	\begin{itemize}
		\item[i.] If $B_1\in \mathcal{Q}(\mathcal{P}_8')$ is such that $a=b=c=d=e=f=g=1$ then $\sigma(B_1)=\{-1.3096 + 0.0611i, -1.3096 - 0.0611i, 1.3096 + 0.0611i, 1.3096 - 0.0611i, 1.5080i, - 1.5080i,  -0.3858i,0.3858i\}$, so $In(B_1)=(2,2,4)$ and $B_1$ has all distinct eigenvalues.
		\item[ii.] If $B_2\in \mathcal{Q}(\mathcal{P}_8')$ is such that $a=g=\frac{1}{20}$, $b=d=f=5$, $c=e=20$ then $\sigma(B_2)=\{ 5.3989i,-5.3989i, 2.2575i, -2,2575i, 0.1022i,-0.1022i,0.8032i,0.8032i\}$, so $In(B_2)=(0,0,8)$ and $B_2$ has all distinct eigenvalues.
	\end{itemize}
\end{eg}

\begin{theorem}\label{nth8} \label{nth3.18}
	Let $\mathcal{P}\in \mathcal{Q}_n$ be an irreducible tridiagonal sign pattern with a $0$-diagonal. If $\mathcal{P}_6'$ or ${\mathcal{P}_6'}_-$ is a submatrix of $\mathcal{P}$, then $\mathcal{P}$ does not require a unique inertia.
\end{theorem}

\begin{proof}
	Suppose that $\mathcal{P}=[p_{ij}]\in \mathcal{Q}_n$ is an irreducible tridiagonal sign pattern with a $0$-diagonal such that $\mathcal{P}_6'$ is a submatrix of $\mathcal{P}$, thus $n\geq 6$. Let $G$ be the signed undirected graph corresponding to $\mathcal{P}$. If the number of maximal signed paths of odd length in $G$ is more than one, then by Theorem \ref{th2.51}, $\mathcal{P}$ does not require a unique inertia, so we assume that the number of maximal odd length paths in $G$ is exactly one, which implies that $n$ is even.
	
	For $n=6$, by Example \ref{ega21}, $\mathcal{P}$ does not require a unique inertia. For $n=8$, since the number of maximal signed paths of odd length in $G$ is one, $\mathcal{P}$ is equivalent to $\mathcal{P}_8'$. By Example \ref{ega11}, $\mathcal{P}$ does not require a unique inertia, and there exists $B_1, B_2\in \mathcal{Q}(\mathcal{P})$ such that $\In(B_1)\neq \In(B_2)$ and both $B_1, B_2$ have all distinct eigenvalues.
	
	For $n=2k$ with $k\geq5$, since the number of maximal signed paths of odd length in $G$ is one, $\mathcal{P}_8'$ is a submatrix of $\mathcal{P}$ and by proceeding as in Theorem \ref{nth3.4}, it can be shown that $\mathcal{P}$ does not require a unique inertia.
	
	Similarly, if ${\mathcal{P}_{6-}'}$ is a submatrix of $\mathcal{P}$, then $\mathcal{P}$ does not require a unique inertia.
\end{proof}


\section{Combinatorially symmetric sign patterns with underlying graphs containing cycles, which require a unique inertia} \label{s4}

In this section, we focus on irreducible combinatorially symmetric sign patterns, where the underlying undirected graphs contain cycles but no loops. Interesting and easily verifiable necessary conditions are obtained for such patterns to require a unique inertia, depending on the signs of the edges, the distances between cycles, and the distances from leaves to cycles in the underlying graph.

\begin{theorem}\label{nth0.11}
	Let $\mathcal{P}\in \mathcal{Q}_n$, where $n$ is odd, be an irreducible combinatorially symmetric sign pattern with a $0$-diagonal whose underlying signed undirected graph is a cycle. Then the following is true.
	\begin{itemize}
		\item[i.] If $\mathcal{P}$ allows singularity, then $\mathcal{P}$ does not require a unique inertia.
		\item[ii.] If $\mathcal{P}$ is sign nonsingular, then $\mathcal{P}$ requires a unique inertia. 
	\end{itemize}
\end{theorem}
\begin{proof}
	\begin{itemize}
		\item[i.] If $\mathcal{P}$ allows singularity, and since $\mathcal{P}$ is not sign singular, then there exist two composite cycles $\Gamma_1$, $\Gamma_2$ of length $n$ in the underlying signed directed graph of $\mathcal{P}$ such that $\sgn(\Gamma_1) = -\sgn(\Gamma_2)$.
		Hence, by Lemma~\ref{xnl1}, $\mathcal{P}$ does not require a unique inertia.
		\item[ii.] If $\mathcal{P}$ is sign nonsingular, then for any $B\in \mathcal{Q}(\mathcal{P})$, the characteristic polynomial of $B$ has the form $$\Ch_B(x)=x^n+E_2(B)x^{n-2}+\cdots+E_{n-1}(B)x+(-1)^nE_n(B),$$ where $E_n(B)\neq 0$. Since the only non-vanishing term of even power of $x$ in $\Ch_B(x)$ is $E_n(B)~(\neq 0)$, $i\lambda$ is not a root of $\Ch_B(x)=0$ for any $\lambda\in\mathbb{R}$, so $B$ has no eigenvalues with zero real part. Hence, by Theorem \ref{th4}, $\mathcal{P}$ requires a unique inertia.
	\end{itemize}
\end{proof}

\begin{eg}\label{xxeg1}
	Suppose that $\mathcal{P}\in \mathcal{Q}_3$ is an irreducible combinatorially symmetric sign pattern with a $0$-diagonal, whose underlying signed undirected graph $G$ is given in Fig.~\ref{xx1}.
	\begin{figure}[H]
		\centering
		\begin{minipage}{.45\textwidth}
			\vspace{-1cm}
			\[
			\mathcal{P}=
			\begin{bmatrix}
				0 & + & + \\
				+ & 0 & + \\
				- & + & 0
			\end{bmatrix}
			\]
		\end{minipage}
		\hspace{0.05\textwidth}
		\begin{minipage}{.45\textwidth}
			\centering
			\tikzset{node distance=1cm}
			\begin{tikzpicture}[
				vertex/.style={draw, circle, inner sep=0pt, minimum size=.15cm, fill=black},
				edge/.style={thick}
				]
				\node[vertex,label=left:3]  (v1) at(-1,0) {};
				\node[vertex,label=right:2] (v2) at( 1,0) {};
				\node[vertex,label=above:1] (v3) at( 0,1) {};
				
				\draw[edge] (v1)--node[below]{$+$}(v2);
				\draw[edge] (v2)--node[right]{$+$}(v3);
				\draw[edge] (v3)--node[left] {$-$}(v1);
			\end{tikzpicture}
			\caption{ $G$.}
			\label{xx1}
		\end{minipage}
	\end{figure}
	
	Then $\Gamma_1 = p_{12}p_{23}p_{31}$ and $\Gamma_2 = p_{13}p_{32}p_{21}$ are two oppositely signed cycles of length $3$ in the signed directed graph of $\mathcal{P}$. Hence, $\mathcal{P}$ allows singularity. Therefore, $\mathcal{P}$ satisfies the conditions stated in (i) of Theorem~\ref{nth0.11}, which implies that $\mathcal{P}$  does not require a unique inertia.
\end{eg}
\begin{eg}
	Consider an irreducible combinatorially symmetric sign pattern $\mathcal{P}\in \mathcal{Q}_3$ with a $0$-diagonal whose underlying signed undirected graph $G$ is given in Fig.~\ref{xx2}.
	\begin{figure}[H]
		\centering
		\begin{minipage}{.45\textwidth}
			\vspace{-1cm}
			\[
			\mathcal{P}=
			\begin{bmatrix}
				0 & + & + \\
				+ & 0 & + \\
				+ & + & 0
			\end{bmatrix}
			\]
		\end{minipage}
		\hspace{0.05\textwidth}
		\begin{minipage}{.45\textwidth}
			\centering
			\tikzset{node distance=1cm}
			\begin{tikzpicture}[
				vertex/.style={draw, circle, inner sep=0pt, minimum size=.15cm, fill=black},
				edge/.style={thick}
				]
				\node[vertex,label=left:3]  (v1) at(-1,0) {};
				\node[vertex,label=right:2] (v2) at( 1,0) {};
				\node[vertex,label=above:1] (v3) at( 0,1) {};
				
				\draw[edge] (v1)--node[below]{$+$}(v2);
				\draw[edge] (v2)--node[right]{$+$}(v3);
				\draw[edge] (v3)--node[left] {$+$}(v1);
			\end{tikzpicture}
			\caption{$G$}
			\label{xx2}
		\end{minipage}
	\end{figure}
	
	Then $\Gamma_1 = p_{12}p_{23}p_{31}$ and $\Gamma_2 = p_{13}p_{32}p_{21}$ are the only composite cycles of length $3$ in the signed directed graph of $\mathcal{P}$, and they have the same sign. Hence, $\mathcal{P}$ is sign nonsingular. Therefore, $\mathcal{P}$ satisfies the conditions stated in (ii) of Theorem~\ref{nth0.11}, which implies that $\mathcal{P}$ requires a unique inertia.
\end{eg}

\begin{lemma}\label{xnl2}
	Let $\mathcal{P} \in \mathcal{Q}_n$ be a sign pattern, whose underlying signed directed graph $D$ contain only cycles of even length. Suppose that the maximum length of the composite cycles in $D$ is $m$, $m \geq 2$. If $D$ contains two composite cycles, $\Gamma_1$, $\Gamma_2$, consisting only of negative $2$-cycles and positive $2$-cycles, respectively, and if $l(\Gamma_1)+l(\Gamma_2)\geq m+2$, then $\mathcal{P}$ does not require a unique inertia.
\end{lemma}
\begin{proof}
	Let $\Gamma_1 = \alpha_{i_1} \alpha_{i_2} \cdots \alpha_{i_{k_1}}$ and $\Gamma_2 = \alpha_{j_1} \alpha_{j_2} \cdots \alpha_{j_{k_2}}$ be two composite cycles in $D$ such that $\Gamma_1$ consists only of negative $2$-cycles and $\Gamma_2$ consists only of positive $2$-cycles, with $l(\Gamma_1)+l(\Gamma_2)=2k_1+2k_2\geq m+2$, where each $\alpha_t$ denotes a $2$-cycle.   Define the real matrix $B_{\Gamma_1}(0)=[b_{\Gamma_1}(0)_{ij}]$ as
	$$b_{\Gamma_1}(0)_{ij}= \begin{cases}
		10^p & \text{if}\ p_{ij}=+~\text{and}~ \text{is}~
		\text{in}~ \alpha_p,~\text{for}~p=i_1,i_2,...,i_{k_1} \\
		-10^p & \text{if}\ p_{ij}=-~\text{and}~ \text{is}~ \text{in}~ \alpha_p,~\text{for}~p=i_1,i_2,...,i_{k_1} \\
		
		0 & \text{elsewhere.}
	\end{cases}$$
	Then the eigenvalues of $B_{\Gamma_1}(0)$ are the second complex roots of $-10^{2p}$ for $p=i_1,i_2,...,i_{k_1}$. Therefore, 
	$B_{\Gamma_1}(0)$ has $2k_1$ algebraically simple nonzero purely imaginary eigenvalues. For $\epsilon>0$, sufficiently small define $B_{\Gamma_1}(\epsilon)$ as in \eqref{e2} with $\gamma$ replaced by $\Gamma_1$, then $B_{\Gamma_1}(\epsilon)\in \mathcal{Q}(\mathcal{P})$ and its eigenvalues remain close to the eigenvalues of $B_{\Gamma_1}(0)$. Also, $D$ contains only cycles of even length, so the characteristic polynomial $\Ch_{B_{\Gamma_1}(\epsilon)}(x)$ consists only of even powers of $x$, so by Lemma \ref{lem1}, $i_0(B_{\Gamma_1}(\epsilon))\geq (n-m)+2k_1$.
	
	Similarly, define the real matrix $B_{\Gamma_2}(0)$, then it has $2k_2$ algebraically simple, nonzero real eigenvalues. For $\epsilon'>0$, sufficiently small define $B_{\Gamma_2}(\epsilon')$ as in \eqref{e2} with $\gamma$ replaced by $\Gamma_2$, then $B_{\Gamma_2}(\epsilon')\in \mathcal{Q}(\mathcal{P})$ and the eigenvalues are remain close to the eigenvalues of $B_{\Gamma_2}(0)$. So, $n-i_0(B_{\Gamma_2}(\epsilon'))\geq 2k_2$, and $$i_0(B_{\Gamma_1}(\epsilon))+(n-i_0(B_{\Gamma_2}(\epsilon')))\geq (n-m)+2k_1+2k_2\geq n+2.$$ Hence, by Theorem \ref{lem2}, $\mathcal{P}$ does not require a unique inertia.
\end{proof}

\begin{theorem} \label{xxth3}
	Let $\mathcal{P} \in \mathcal{Q}_n$ be an irreducible combinatorially symmetric sign pattern with a $0$-diagonal, whose underlying signed undirected graph $G$ is a cycle. Then $\mathcal{P}$ does not require a unique inertia if any one of the following holds:
	\begin{itemize}
		\item[i.] $G$ has an odd number of negative edges. 
		\item[ii.] Every edge of $G$ is negative.
		\item[iii.] $n$ is even and $G$ has a maximal signed path of odd length.
	\end{itemize}
\end{theorem}

\begin{proof}
	Let $D$ be the underlying signed directed graph of $\mathcal{P}$.
	\begin{itemize}
		\item[i.]  If $G$ contains an odd number of negative edges, then
		$$p_{12}p_{23}\cdots p_{n-1,n} p_{n1}=-p_{1n}p_{n,n-1}\cdots p_{32} p_{21},$$ where $p_{ij}$ is the $(i,j)$-th entry of $\mathcal{P}$.
		Let $\Gamma_1 = p_{12}p_{23}\cdots p_{n-1,n} p_{n1}$, $\Gamma_2 = p_{1n}p_{n,n-1}\cdots p_{32} p_{21}$.
		Then $\Gamma_1$, $\Gamma_2$ are two oppositely signed cycles in $D$ of length $n$, so by Lemma~\ref{xnl1}, $\mathcal{P}$ does not require a unique inertia.
		
		\item[ii.] If $n$ is odd and all edges are negatively signed, then by part~(i), $\mathcal{P}$ does not require a unique inertia. Therefore, suppose that $n$ is even. Let 
		$\Gamma_1=(1,2)(3,4)\cdots (n-1,n)$, $\Gamma_2=(1,2,...,n)$ be two composite cycles in $D$, each of length $n$.  
		Define a real matrix $B_{\Gamma_1}(\epsilon) = [b_{\Gamma_1}(\epsilon)_{ij}] \in \mathcal{Q}(\mathcal{P})$ by
		\begin{equation*}
			|b_{\Gamma_1}(\epsilon)_{ij}|= \begin{cases}
				10^p & \text{if}\ p_{ij}\neq 0 ~\text{and~is~in}~ (p,p+1),~\text{for}~p=1,3,...,n-1 \\
				\epsilon & \text{if}\ p_{ij}\neq 0~\text{and~not~in}~ \Gamma_1 \\
				
				0 & \text{elsewhere,}
			\end{cases}
		\end{equation*}
		for some $\epsilon>0$. Then $B_{\Gamma_1}(\epsilon)$ is a real skew-symmetric matrix, and hence $\In(B_{\Gamma_1}(\epsilon))=(0,0,n)$. Also, define two real matrices $B_{\Gamma_2}(0)$ and $B_{\Gamma_2}(\epsilon)$ as in \eqref{e1} and \eqref{e2}, respectively, with $\gamma$ replaced by $\Gamma_2$. Since $B_{\Gamma_2}(0)$ has $n$ nonzero algebraically simple eigenvalues, which are the $n$-th complex roots of $+1$ or $-1$ $(n\geq 3)$, it follows that for $\epsilon > 0$ sufficiently small, $B_{\Gamma_2}(\epsilon)\in \mathcal{Q}(\mathcal{P})$ and it has $n$ nonzero algebraically simple eigenvalues close to those of $B_{\Gamma_2}(0)$. Therefore, $\In(B_{\Gamma_2}(\epsilon))\neq (0,0,n)$, and $\mathcal{P}$ does not require a unique inertia. 
		
		\item[iii.] Suppose that $n$ is even and $G$ has a maximal signed path of odd length. Without loss of generality, let $P=\{1,2\}\{2,3\}\cdots\{2p-1,2p\}$ be the maximal signed path in $G$ of odd length. Since $n$ is even, the path from $2p$ to $1$ is of odd length.
		Consider,
		$$\Gamma=\{\{2j-1,2j\}: 1\leq j \leq p\}~\cup~\{\{2i,2i+1\}:2p\leq 2i\leq (n-2)\}~\cup~\{\{n,1\}\}. $$ 
		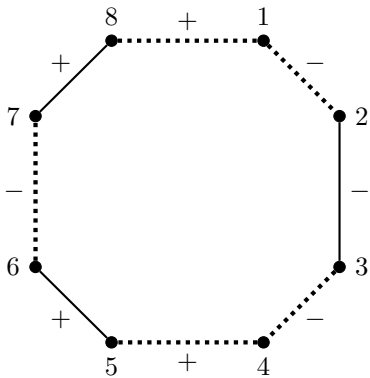
\begin{figure}[H]
			\tikzset{node distance=1cm}
			\centering
			\begin{tikzpicture}
				\tikzstyle{vertex}=[draw, circle, inner sep=0pt, minimum size=.15cm, fill=black]
				\tikzstyle{edge}=[,thick]
				\node[vertex,label=above:1](v1)at(1,2){};
				\node[vertex,label=right:2](v2)at(2,1){};
				\node[vertex,label=right:3](v3)at(2,-1){};
				\node[vertex,label=below:4](v4)at(1,-2){};
				\node[vertex,label=below:5](v5)at(-1,-2){};
				\node[vertex,label=left:6](v6)at(-2,-1){};
				\node[vertex,label=left:7](v7)at(-2,1){};
				\node[vertex,label=above:8](v8)at(-1,2){};
				\draw[edge, ultra thick, dotted](v1)--node[right, yshift=0.2cm, xshift=-0.1cm]{$-$}(v2);
				\draw[edge](v2)--node[right]{$- $}(v3);
				\draw[edge, ultra thick, dotted](v3)--node[right, yshift=-0.2cm, xshift=-0.1cm]{$-$}(v4);
				\draw[edge, ultra thick, dotted](v4)--node[below]{$+$}(v5);
				\draw[edge](v5)--node[left, yshift=-0.2cm, xshift=0.1cm]{$+$}(v6);
				\draw[edge, ultra thick, dotted](v6)--node[left]{$-$}(v7);
				\draw[edge](v7)--node[left, yshift=0.2cm, xshift=0.1cm]{$ +$}(v8);
				\draw[edge, ultra thick, dotted](v8)--node[above]{$+$}(v1);
			\end{tikzpicture}\caption{ In the above figure $G$ is a cycle of length $8$ and $P=\{1,2\}\{2,3\}\{3,4\}$, where the dotted edges correspond to $\Gamma$.} \end{figure}
		Then the edges $\{2p-1,2p\}$, $\{2p,2p+1\}$; $\{n,1\}$, $\{1,2\}$ are the only pair of edges in $\Gamma$ which are adjacent, and the edges in each pair are signed oppositely. Hence, the positive edges in $\Gamma$, the negative edges in $\Gamma$ are all non-adjacent. Let $k_1$ (respectively, $k_2$) be the number of negative (respectively, positive) edges in $\Gamma$. Since every vertex of $G$ occurs once in $\Gamma$ except $2p$, $1$, which occurs twice, hence $2k_1+2k_2=n+2$. 
		Let $\mathcal{M}_1=\{\alpha_{i_1},\alpha_{i_2},..., \alpha_{i_{k_1}}\}$ and $\mathcal{M}_2=\{\alpha_{j_1},\alpha_{j_2},..., \alpha_{j_{k_2}}\}$ be two matchings in $G$ containing only negative and positive edges in $\Gamma$, respectively, where $\alpha_t=\{t,t+1\}$ for $t=1,2,...,n-1$, and $\alpha_n=\{n,1\}$. 
		
		Let $\Gamma_1$, $\Gamma_2$ be the composite cycles in $D$ corresponding to $\mathcal{M}_1$, $\mathcal{M}_2$, respectively. Then, $\Gamma_1$ (respectively, $\Gamma_2$) consisting only of negative (respectively, positive) $2$-cycles of $D$, and since $l(\Gamma_1) + l(\Gamma_2) = n + 2$, it follows from Lemma~\ref{xnl2} that $\mathcal{P}$ does not require a unique inertia. 
	\end{itemize} 
\end{proof}


The following example shows that if $\mathcal{P}$ does not satisfy any of the three conditions of the previous theorem, then we cannot conclude either way, that is, in that case, $\mathcal{P}$ may or may not require a unique inertia.

\begin{eg}\label{eg0.6}
	Suppose that $\mathcal{P}\in \mathcal{Q}_4$ is an irreducible combinatorially symmetric sign pattern with a $0$-diagonal, whose underlying signed undirected graph $G$ is given in Fig.\ref{fign1}.
	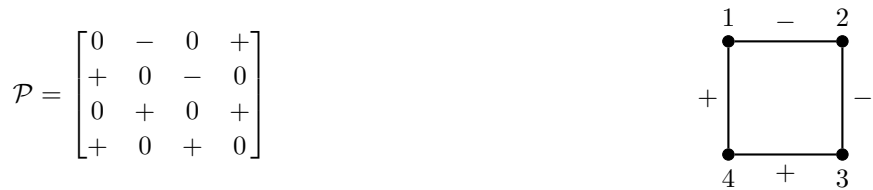
\begin{figure}[H]
		\centering
		\begin{minipage}{.45\textwidth}
			\vspace{-1cm}
			\[
			\mathcal{P}=
			\begin{bmatrix}
				0 & - & 0 & + \\
				+ & 0 & - & 0 \\
				0 & + & 0 & + \\
				+ & 0 & + & 0
			\end{bmatrix}
			\]
		\end{minipage}
		\hspace{0.05\textwidth}
		\begin{minipage}{.45\textwidth}
			\centering
			\tikzset{node distance=1cm}
			\begin{tikzpicture}[
				vertex/.style={draw, circle, inner sep=0pt, minimum size=.15cm, fill=black},
				edge/.style={thick}
				]
				\node[vertex,label=above:1] (v1) at(-.75,.75) {};
				\node[vertex,label=above:2] (v2) at( .75,.75) {};
				\node[vertex,label=below:3] (v3) at( .75,-.75) {};
				\node[vertex,label=below:4] (v4) at(-.75,-.75) {};
				
				\draw[edge] (v1)--node[above] {$-$}(v2);
				\draw[edge] (v2)--node[right] {$-$}(v3);
				\draw[edge] (v3)--node[below] {$+$}(v4);
				\draw[edge] (v1)--node[left]  {$+$}(v4);
			\end{tikzpicture}
			\caption{Signed undirected graph $ G$ of $\mathcal{P}$.}
			\label{fign1}
		\end{minipage}
	\end{figure}
	
	Then $G$ is a cycle with exactly two negative edges and no maximal signed path of odd length. Therefore, $\mathcal{P}$ does not satisfy any of the three conditions
	of Theorem~\ref{xxth3}. Suppose that
	$$B=\begin{bmatrix}
		0&-a&0&b\\c&0&-d&0\\0&e&0&f\\g&0&h&0
	\end{bmatrix}\in \mathcal{Q}(\mathcal{P}),~\text{where}~a,b,c,d,e,f,g,h>0.$$
	Then the characteristic polynomial of $B$ is $\Ch_B(x)=x^4+(ac-bg+ed-fh)x^2-(acfh+adfg+ebch+ebdg)$. Take $f(x)=x^2+(ac-bg+ed-fh)x-(acfh+adfg+ebch+ebdg)$, then
	$$ 
	\sgn(f(x)) =  \begin{cases} 
		(+)-(*)(+)-(+) & \text{if} ~ x> 0 \\
		(+)-(*)(-)-(+) & \text{if} ~x<0,
	\end{cases} 
	$$ where $*$ denotes an arbitrary sign. Since $V_+(x)=1$ and $V_-(x)=1$, according to Descartes’ rule of signs, the number of positive and negative real roots of $f(x)$ is exactly one. Also, $\Ch_{B}(x)$ consists only of even powers of $x$, so by Lemma \ref{lem1}, $\In(B)=(1,1,2)$ for all $B\in \mathcal{Q}(\mathcal{P})$ and $\mathcal{P}$ requires a unique inertia.
\end{eg}

\begin{eg}\label{egxx0.6} Suppose that $\mathcal{P}\in \mathcal{Q}_4$ is an irreducible combinatorially symmetric sign pattern with a $0$-diagonal, whose underlying signed undirected graph is the graph $G$ given in Fig~\ref{xfig3.37}.
	\begin{figure}[H]
		\centering
		\begin{minipage}{.45\textwidth}
			\vspace{-1cm}
			\[
			\mathcal{P}=
			\begin{bmatrix}
				0 & + & 0 & + \\
				+ & 0 & + & 0 \\
				0 & + & 0 & + \\
				+ & 0 & + & 0
			\end{bmatrix}
			\]
		\end{minipage}
		\hspace{0.05\textwidth}
		\begin{minipage}{.45\textwidth}
			\centering
			\tikzset{node distance=1cm}
			\begin{tikzpicture}[
				vertex/.style={draw, circle, inner sep=0pt, minimum size=.15cm, fill=black},
				edge/.style={thick}
				]
				\node[vertex,label=above:1] (v1) at(-.75,.75) {};
				\node[vertex,label=above:2] (v2) at( .75,.75) {};
				\node[vertex,label=below:3] (v3) at( .75,-.75) {};
				\node[vertex,label=below:4] (v4) at(-.75,-.75) {};
				
				\draw[edge] (v1)--node[above] {$+$}(v2);
				\draw[edge] (v2)--node[right] {$+$}(v3);
				\draw[edge] (v3)--node[below] {$+$}(v4);
				\draw[edge] (v1)--node[left]  {$+$}(v4);
			\end{tikzpicture}
			\caption{Signed undirected graph $ G$ of $\mathcal{P}$.}
			\label{xfig3.37}
		\end{minipage}
	\end{figure}
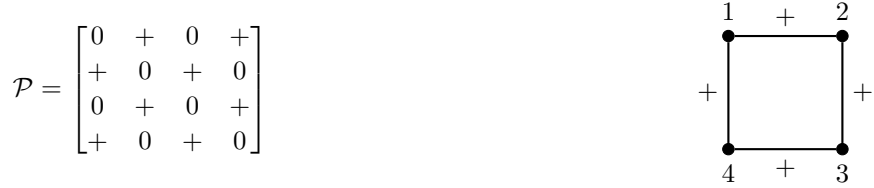
	
	Then $G$ is a cycle with no negative edges and no maximal signed path of odd length. Therefore, $\mathcal{P}$ does not satisfy any of the three conditions
	of Theorem~\ref{xxth3}. However, consider	$$\begin{bmatrix}
		0&a&0&b\\ c&0&d&0\\ 0&e&0&f\\ g&0&h&0
	\end{bmatrix}\in \mathcal{Q}(\mathcal{P}), ~\text{where}~ a,b,c,d,e,f,g,h>0.$$ \begin{itemize}
		\item[i.] Let $B_1 \in \mathcal{Q}(\mathcal{P})$ with $a = b = c = d = e = f = g = h = 1$. Then $\sigma(B_1) = \{0, 0, 2, -2\}$, and $\In(B_1) = (1, 1, 2)$.
		\item[ii.] Let $B_2 \in \mathcal{Q}(\mathcal{P})$ with $a = c = f = h = 2$ and $b=d=e=g = 1$. Then $\sigma(B_2) = \{3, 1, -1, -3\}$, and $\In(B_2) = (2, 2, 0)$.
	\end{itemize}
	
	Therefore, $\mathcal{P}$ does not require a unique inertia.
\end{eg}
The following example shows that if in condition (iii) of the previous theorem, $\mathcal{P}$ is of odd order, then it may require a unique inertia even if it contains a maximal signed path of odd length.

\begin{eg}\label{xeg1} Suppose that $\mathcal{P}\in \mathcal{Q}_3$ is an irreducible combinatorially symmetric sign patterns with a $0$-diagonal, whose underlying signed undirected graph is $G$ given in Fig.\ref{xfign1.9}.
	\begin{figure}[H]
		\centering
		\begin{minipage}{.45\textwidth}
			\vspace{-1cm}
			\[
			\mathcal{P}=
			\begin{bmatrix}
				0 & - & + \\
				+ & 0 & - \\
				+ & + & 0
			\end{bmatrix}
			\]
		\end{minipage}
		\hspace{0.05\textwidth}
		\begin{minipage}{.45\textwidth}
			\centering
			\tikzset{node distance=1cm}
			\begin{tikzpicture}[
				vertex/.style={draw, circle, inner sep=0pt, minimum size=.15cm, fill=black},
				edge/.style={thick}
				]
				\node[vertex,label=left:1]  (v1) at(-1,0) {};
				\node[vertex,label=right:2] (v2) at( 1,0) {};
				\node[vertex,label=above:3] (v3) at( 0,1) {};
				
				\draw[edge] (v1)--node[below] {$-$}(v2);
				\draw[edge] (v2)--node[right] {$-$}(v3);
				\draw[edge] (v3)--node[left]  {$+$}(v1);
			\end{tikzpicture}
			\caption{Signed undirected graph $ G$ of $\mathcal{P}$.}
			\label{xfign1.9}
		\end{minipage}
	\end{figure}
	
	Then $G$ is a cycle containing a maximal signed path of odd length. However, $\mathcal{P}$ is a sign nonsingular sign pattern of odd order, by Theorem \ref{nth0.11}, $\mathcal{P}$ requires a unique inertia.
\end{eg}

The following are examples of sign patterns that satisfy conditions (i), (ii), and (iii) of Theorem~\ref{xxth3}, respectively, and therefore do not require a unique inertia.

\begin{eg}
	Let $\mathcal{P}\in \mathcal{Q}_4$ be an irreducible combinatorially symmetric sign pattern with a $0$-diagonal, whose underlying undirected graph is $G$ as in Example~\ref{xxeg1}. Since $G$ is a cycle with exactly one negative edge, so by Theorem~\ref{xxth3}(i), $\mathcal{P}$ does not require a unique inertia.
\end{eg}

\begin{eg} Suppose that $\mathcal{P}\in \mathcal{Q}_4$ is an irreducible combinatorially symmetric sign pattern with a $0$-diagonal, whose underlying signed undirected graph is the graph $G$ given in Fig~\ref{xfig3.3}.
	\begin{figure}[H]
		\centering
		\begin{minipage}{.45\textwidth}
			\vspace{-1cm}
			\[
			\mathcal{P}=
			\begin{bmatrix}
				0 & + & 0 & - \\
				- & 0 & + & 0 \\
				0 & - & 0 & + \\
				+ & 0 & - & 0
			\end{bmatrix}
			\]
		\end{minipage}
		\hspace{0.05\textwidth}
		\begin{minipage}{.45\textwidth}
			\centering
			\tikzset{node distance=1cm}
			\begin{tikzpicture}[
				vertex/.style={draw, circle, inner sep=0pt, minimum size=.15cm, fill=black},
				edge/.style={thick}
				]
				\node[vertex,label=above:1] (v1) at(-.75,.75) {};
				\node[vertex,label=above:2] (v2) at( .75,.75) {};
				\node[vertex,label=below:3] (v3) at( .75,-.75) {};
				\node[vertex,label=below:4] (v4) at(-.75,-.75) {};
				
				\draw[edge] (v1)--node[above] {$-$}(v2);
				\draw[edge] (v2)--node[right] {$-$}(v3);
				\draw[edge] (v3)--node[below] {$-$}(v4);
				\draw[edge] (v1)--node[left]  {$-$}(v4);
			\end{tikzpicture}
			\caption{Signed undirected graph $ G$ of $\mathcal{P}$.}
			\label{xfig3.3}
		\end{minipage}
	\end{figure}
	
	Then $G$ is a cycle with all edges negatively signed, so by Theorem~\ref{xxth3}(ii), $\mathcal{P}$ does not require a unique inertia.
\end{eg}

\begin{eg}\label{xxeg3.9} Suppose that $\mathcal{P}\in \mathcal{Q}_4$ is an irreducible combinatorially symmetric sign pattern with a $0$-diagonal, whose underlying signed undirected graph is $G$ given in Fig.\ref{xnfig2}.
	
	\begin{figure}[H]
		\centering
		\begin{minipage}{.45\textwidth}
			\vspace{-1cm}
			\[
			\mathcal{P}=
			\begin{bmatrix}
				0 & + & 0 & + \\
				- & 0 & + & 0 \\
				0 & + & 0 & + \\
				+ & 0 & - & 0
			\end{bmatrix}
			\]
		\end{minipage}
		\hspace{0.05\textwidth}
		\begin{minipage}{.45\textwidth}
			\centering
			\tikzset{node distance=1cm}
			\begin{tikzpicture}[
				vertex/.style={draw, circle, inner sep=0pt, minimum size=.15cm, fill=black},
				edge/.style={thick}
				]
				\node[vertex,label=above:1] (v1) at(-.75,.75) {};
				\node[vertex,label=above:2] (v2) at( .75,.75) {};
				\node[vertex,label=below:3] (v3) at( .75,-.75) {};
				\node[vertex,label=below:4] (v4) at(-.75,-.75) {};
				
				\draw[edge] (v1)--node[above] {$+$}(v2);
				\draw[edge] (v2)--node[right] {$-$}(v3);
				\draw[edge] (v3)--node[below] {$+$}(v4);
				\draw[edge] (v1)--node[left]  {$-$}(v4);
			\end{tikzpicture}
			\caption{Signed undirected graph $ G$ of $\mathcal{P}$.}
			\label{xnfig2}
		\end{minipage}
	\end{figure}
	Then $G$ is a cycle with a maximal signed path of odd length, so by Theorem~\ref{xxth3}(iii), $\mathcal{P}$ does not require a unique inertia.
\end{eg}

We use the following lemma to determine specific patterns or structures that an irreducible combinatorially symmetric sign pattern with a $0$-diagonal containing cycles cannot have if it requires a unique inertia.


\begin{lemma}\label{l28}
	
	Let $\mathcal{P}\in \mathcal{Q}_n$ be an irreducible combinatorially symmetric sign pattern with a $0$-diagonal, whose underlying signed undirected graph is $G$ and the signed directed graph is $D$. Suppose that $G$ has exactly one cycle $\mathcal{C}=u_1u_2\cdots u_k$ of length $k$ and the distance of $\mathcal{C}$ from any leaf in $G$ is even. Then the maximum length of a composite cycle in $D$ is equal to the sum of the length of $\mathcal{C}$ and the maximum length of a composite cycle in $D\setminus V(\mathcal{C})$.       
\end{lemma}
\begin{proof} Let $D'=D\setminus V(\mathcal{C})$. If $G$ has no leaf, then there is nothing to prove. Suppose that $G$ has at least one leaf, then we prove this result by using induction on $m$, where $m$ is the number of leaves in $G$. For $m=1$, let $P=\{v_0,v_1\}\{v_1,v_2\}\cdots \{v_{2r-1},v_{2r}\}$ be the path from a vertex $v_0$ of $\mathcal{C}$ to the leaf $v_{2r}$ in $G$. Then $\bar{\Gamma}=\gamma\gamma_1\gamma_3\cdots\gamma_{2r-1},$ is a composite cycle with maximum length in $D$, 
	where $\gamma$ is the directed cycle $(u_1u_2\cdots u_k)$ and each $\gamma_i$ is a directed $2$-cycle $(v_i,v_{i+1})$ for $i=1,3,...,2r-1$. Consider
	$$\Gamma=\gamma_1\gamma_3\cdots\gamma_{2r-1},$$ 
	then $\Gamma$ is a composite cycle with maximum length in $D'$. So, the result is true for $m=1$.
	
	Suppose that the result holds
	for all $m$, $m \leq t-1$, $t \geq 2$.
	For $m = t$, let $u$ be a leaf in $G$ such that $\dist(u,\mathcal{C})$ is maximum. Let $v$ be the unique neighbour of $u$, then $v\notin V(\mathcal{C})$.
	
	\textbf{Case-1:} $v$ is adjacent to more than one leaf in $G$.\\
	Let $P_1=\{v_0,v_1\}\{v_1,v_2\}\cdots \{v_{2r},v\}\{v,u\}$ be a path from a vertex $v_0$ of $\mathcal{C}$ to $u$ such that none of $v_1,v_2,...,v_{2r}$ is in $V(\mathcal{C})$ and let $w$ be a neighbour of $v$ which is not in $P_1$. 
	\begin{figure}[H]
		\tikzset{node distance=1cm}
		\centering
		\begin{tikzpicture}
			\tikzstyle{vertex}=[draw, circle, inner sep=0pt, minimum size=.15cm, fill=black]
			\tikzstyle{edge}=[,thick]
			\node[vertex,label=below:$~v_0$](v1)at(0,0){};
			\node[vertex,label=below:$v_1$](v7)at(1,0){};
			\node[vertex,label=below:$v_2$](v8)at(2,0){};
			\node[vertex,label=below:$v_{2r}$](v9)at(3,0){};
			\node[vertex,label=below:$v$](v10)at(4,0){};
			\node[vertex,label=below:$u$](v11)at(5,0){};
			\node[vertex,label=below:$w$](v12)at(5,-1){};
			\node[vertex](v13)at(-0.75,1){};
			\node[vertex](v14)at(-0.75,-1){};
			\node[vertex] (v15) at (5,1) {};
			\node[vertex] (v16) at (4,1) {};
			\draw[thick]   (0,1.5) to [bend left=35] node[left] {$ $} (0,-1.5);
			\draw[edge](v1)--node[]{$ $}(v7);
			\draw[edge](v7)--node[]{$ $}(v8);
			\draw[thick, dotted](v8)--node[]{$ $}(v9);
			\draw[edge](v9)--node[]{$ $}(v10);
			\draw[edge](v10)--node[]{$ $}(v11);
			\draw[edge](v10)--node[]{$ $}(v12);
			\draw[thick, dotted](v1)--node[]{$ $}(v13);
			\draw[thick, dotted](v1)--node[]{$ $}(v14);
			\draw[thick, dotted](v10)--node[]{$ $}(v15);
			\draw[thick, dotted](v10)--node[]{$ $}(v16);
			
		\end{tikzpicture}
		\caption{$G$}	
	\end{figure}
	
	Let $\Gamma_1, \Gamma_2$ be composite cycles of maximum length in $D, D'$, respectively. To show $l(\Gamma_1)=k+l(\Gamma_2)$. If $(v,u)\in \Gamma_1~\text{or}~\Gamma_2$, then replace $(v,u)$ by $(v,w)$ in $\Gamma_1,\Gamma_2$. Then $\Gamma_1,\Gamma_2$ are composite cycles in $D\setminus\{u\}$. Since $G\setminus\{u\}$ has $t-1$ leaves and the distance of each leaf to $\mathcal{C}$ in $G\setminus\{u\}$ is even, so by the induction hypothesis $l(\Gamma_1)=k+l(\Gamma_2)$.  If $(v,u)\notin \Gamma_1,\Gamma_2$, then $\Gamma_1$, $\Gamma_2$ are also composite cycles in $D \setminus \{u\}$, and hence the result follows.
	

	
	\textbf{Case-2:} $v$ is adjacent to exactly one leaf in $G$.\\
	Similarly, let $P_1=\{v_0,v_1\}\{v_1,v_2\}\cdots \{v_{2r},v\}\{v,u\}$ be a path from a vertex $v_0$ of $\mathcal{C}$ to $u$ such that none of $v_1,v_2,...,v_{2r}$ are in $V(\mathcal{C})$. 
	\begin{figure}[H]
		\tikzset{node distance=1cm}
		\centering
		\begin{tikzpicture}
			\tikzstyle{vertex}=[draw, circle, inner sep=0pt, minimum size=.15cm, fill=black]
			\tikzstyle{edge}=[,thick]
			\node[vertex,label=below:$~v_0$](v1)at(0,0){};
			\node[vertex,label=below:$v_1$](v7)at(1,0){};
			\node[vertex,label=below:$v_2$](v8)at(2,0){};
			\node[vertex,label=below:$v_{2r}$](v9)at(3,0){};
			\node[vertex,label=below:$v$](v10)at(4,0){};
			\node[vertex,label=below:$u$](v11)at(5,0){};
			\node[vertex](v13)at(-0.75,1){};
			\node[vertex](v14)at(-0.75,-1){};
			
			\draw[thick] (0,1.5) to [bend left=35] node[left] {$ $} (0,-1.5);
			
			\draw[edge](v1)--node[]{$ $}(v7);
			\draw[edge](v7)--node[]{$ $}(v8);
			\draw[thick, dotted](v8)--node[]{$ $}(v9);
			\draw[edge](v9)--node[]{$ $}(v10);
			\draw[edge](v10)--node[]{$ $}(v11);	
			\draw[thick, dotted](v1)--node[]{$ $}(v13);
			\draw[thick, dotted](v1)--node[]{$ $}(v14);
		\end{tikzpicture}
		\caption{$G$}
	\end{figure}
	Let $\Gamma_1, \Gamma_2$ be composite cycles of maximum length in $D, D'$, respectively. Then either $(v,u) \in \Gamma_1$ (respectively, $\Gamma_2$), or $(v_{2r},v) \in \Gamma_1$ (respectively, $\Gamma_2$). If $(v_{2r},v)$ occurs in any one of $\Gamma_1, \Gamma_2$, replace it with $(v,u)$. Let $\Gamma_1'=\Gamma_1\setminus(v,u)$ and $\Gamma_2'=\Gamma_2\setminus(v,u)$. To prove the lemma, it is enough to show that $l(\Gamma_1')=k+l(\Gamma_2')$.


	
	If $v_{2r}$ is not a leaf of $G_1=G\setminus\{u,v\}$, then $G_1$ has $t-1$ leaves. Then, by the induction hypothesis, the result follows.
	If $v_{2r}$ is a leaf of $G_1$, then apply the previous arguments with $D$ replaced by $D\setminus\{u,v\}$.



	
	
	After a finite number of steps, the above process of deleting edges gives a signed undirected graph $G_p$ with fewer than $t$ leaves. Let $D_p$ be the signed directed graph corresponding to $G_p$. By the induction hypothesis, there exists a composite cycle $\Gamma_p$ of maximum length in $D_p\setminus V(\mathcal{C})$ such that the maximum length of a composite cycle in $D_p$ is equal to the sum of the length of $\mathcal{C}$ and the length of $\Gamma_p$, so the result follows. 
\end{proof}

The conclusion of the above theorem is not true if the distance between a leaf and the cycle in $G$ is odd. 
\begin{eg}
	Let $\mathcal{P}\in \mathcal{Q}_6$ be an irreducible combinatorially symmetric sign pattern with a $0$-diagonal, whose underlying undirected graph $G$ is given in Fig.\ref{fig11.1},
	\begin{figure}[H]\label{fig11.1}
		\tikzset{node distance=1cm}
		\centering
		\begin{tikzpicture}
			\tikzstyle{vertex}=[draw, circle, inner sep=0pt, minimum size=.15cm, fill=black]
			\tikzstyle{edge}=[,thick]
			\node[vertex,label=above:$u_1$, ](v1)at(-1,1){};
			\node[vertex,label=below:$u_2$](v2)at(-2,0){};
			\node[vertex,label=below:$u_3$](v3)at(0,0){};
			\node[vertex,label=below:$u_4$](v4)at(1,0){};
			\node[vertex,label=below:$u_5$](v5)at(2,0){};
			\node[vertex,label=below:$u_6$](v6)at(3,0){};
			
			\draw[edge](v1)--node[]{$ $}(v2);
			\draw[edge](v2)--node[]{$ $}(v3);
			\draw[edge](v3)--node[]{$ $}(v1);
			\draw[edge](v3)--node[]{$ $}(v4);
			\draw[edge](v4)--node[]{$ $}(v5);
			\draw[edge](v5)--node[]{$ $}(v6);
		\end{tikzpicture}
		\caption{$G$}
	\end{figure} Then $G$ contains exactly one cycle $\mathcal{C}=u_1u_2u_3$ of length $3$. Let $D$ be the signed directed graph corresponding to $G$. Then the composite cycle $\bar{\Gamma} = (u_1, u_2)(u_3, u_4)(u_5, u_6)$ has length $6$, which is the maximum among the lengths of all composite cycles in $D$ and a composite cycle of maximum length in $D\setminus V(\mathcal{C})$ is $\Gamma=(u_4,u_5)$ of length $2$. Therefore, the maximum length of a composite cycle in $D$ is not equal to the sum of the length of $\mathcal{C}$ and the maximum length of a composite cycle in $D\setminus V(\mathcal{C})$.       
\end{eg}

\begin{theorem}\label{xxth3.14}
	Let $\mathcal{P}\in \mathcal{Q}_n$ be an irreducible combinatorially symmetric sign pattern with a $0$-diagonal, whose underlying signed undirected graph $G$ has exactly one cycle $\mathcal{C}=u_1u_2\cdots u_k$, and the distance of $\mathcal{C}$ from any leaf in $G$ is even. Then $\mathcal{P}$ does not require a unique inertia if any one of the following holds:
	\begin{itemize}
		\item[i.] $\mathcal{C}$ has an odd number of negative edges.
		\item[ii.] Every edge of \(\mathcal{C}\) is negative.
		\item[iii.] $k$ is even and $\mathcal{C}$ has a maximal signed path of odd length.
	\end{itemize}
\end{theorem}
\begin{proof}
	If $G$ contains no leaf, then it follows from Theorem~\ref{xxth3} that $\mathcal{P}$ does not require a unique inertia. Assume that $G$ has at least one leaf. Let $D$ be the signed directed graph corresponding to $G$, and let $\Gamma=\beta_1\beta_2\cdots \beta_l$  be a composite cycle of maximum length in $D\setminus V(\mathcal{C})$, with $l(\Gamma)=2l$, where each $\beta_r$ is a $2$-cycle for $r=1,2,...,l$. Then, by Lemma~\ref{l28}, the maximum length of a composite cycle in $D$ is $m = k+2l$.
	\begin{itemize}
		\item[i.] If $\mathcal{C}$ has an odd number of negative edges, then $\gamma_1=(u_1u_2\cdots u_k)$ and $\gamma_2=(u_ku_{k-1}\cdots u_1)$ are two cycles in $D$ of length $k$ with $\sgn(\gamma_1)=-\sgn(\gamma_2)$. 
		Therefore, 
		$$
		\Gamma_1 = \gamma_1\Gamma \quad\text{and}\quad \Gamma_2 = \gamma_2\Gamma
		$$
		are two composite cycles in $D$ of length $m$ satisfying $\sgn(\Gamma_1)=-\sgn(\Gamma_2)$. 
		By Lemma~\ref{xnl1}, $\mathcal{P}$ does not require a unique inertia.
		
		\item[ii.] If $k$ is odd and all edges of $\mathcal{C}$ are negatively signed, then by part (i), $\mathcal{P}$ does not require a unique inertia. Suppose that $k$ is even, then similarly as in the proof of Theorem~\ref{xxth3}(ii), $D$ has two composite cycles
		$\Gamma_1,\Gamma_2$ of length $k$ and two real matrix $B_{\Gamma_1}(\epsilon) = [b_{\Gamma_1}(\epsilon)_{ij}],B_{\Gamma_2}(\epsilon) = [b_{\Gamma_2}(\epsilon)_{ij}] \in \mathcal{Q}(\mathcal{P}[\{u_1,u_2,...,u_k\}])$ with $\In(B_{\Gamma_1}(\epsilon))\neq \In(B_{\Gamma_2}(\epsilon))$.
		
		Since the eigenvalues of $B_{\Gamma_1}(\epsilon)$ are close to the second complex roots of $-10^{2p}$ for $p=1,3,...,k-1$, and $\Ch_{B_{\Gamma_1}(\epsilon)}(x)$ consists only of the even power of $x$, it follows from Lemma~\ref{lem1} that for $\epsilon>0$ sufficiently small, $B_{\Gamma_1}(\epsilon)$ has $k$ distinct purely imaginary eigenvalues. Also, the eigenvalues of $B_{\Gamma_2}(\epsilon)$ are distinct.  Let,
		$$M>\max\{1,|\lambda|: \lambda~ \text{is an eigenvalue of either} ~B_{\Gamma_1}(\epsilon)~\text{or}~B_{\Gamma_2}(\epsilon)\}.$$ 
		Consider $$\Gamma_1'=\Gamma_1\Gamma \quad\text{and}\quad \Gamma_2'=\Gamma_2\Gamma,$$
		and define the real matrices $B_{\Gamma_r'}(\epsilon)=[b_{\Gamma_r'}(\epsilon)_{ij}]\in \mathcal{Q}(\mathcal{P})$ for $r=1,2$, as
		$$|b_{\Gamma_r'}(\epsilon)_{ij}|= \begin{cases}
			|b_{\Gamma_r}(\epsilon)_{ij}| & \text{if}\ p_{ij}~ \text{is in } \Gamma_r \\
			M^p & \text{if}\ p_{ij}~ \text{is in}~ \beta_p,~\text{for}~p=1,2,...,l \\
			\epsilon & \text{if}~ p_{ij}\neq 0~\text{and not in}~ \Gamma_r, \Gamma\\
			0 & \text{elsewhere.}
		\end{cases}$$
		For $\epsilon>0$ sufficiently small, the eigenvalues of $B_{\Gamma_r'}(\epsilon)$ remain close to the eigenvalues of $B_{\Gamma_r}(\epsilon)$  for $r=1,2$, together with the second complex roots of $M^{2p}$ or $-M^{2p}$ depending on the sign of $\beta_p$, for $p=1,3,...,l$. Since $\In(B_{\Gamma_1}(\epsilon))\neq \In(B_{\Gamma_2}(\epsilon))$ and $D$ contain only cycles of even length, by Lemma~\ref{lem1}, $\In(B_{\Gamma_1'}(\epsilon))\neq \In(B_{\Gamma_2'}(\epsilon))$. Therefore, $\mathcal{P}$ does not require a unique inertia.

		\item[iii.]  If $k$ is even and $\mathcal{C}$ has a maximal signed path of odd length, then similarly as in the proof of Theorem \ref{xxth3}(iii), $\mathcal{C}$ has two matchings $\mathcal{M}_1$ containing $k_1$ number of negative edges and $\mathcal{M}_2$ containing $k_2$ number of positive edges in $G$, such that $2k_1+2k_2=k+2$. Suppose that $\Gamma_1$, $\Gamma_2$ are the composite cycles in $D$ corresponding to $\mathcal{M}_1$, $\mathcal{M}_2$, respectively. Let $\bar{\Gamma}_1, \bar{\Gamma}_2$ be two composite cycles in $D\setminus V(\mathcal{C})$ containing only negative, positive $2$-cycles of $\Gamma$, respectively. Consider $$\Gamma_1'=\Gamma_1\bar{\Gamma}_1\quad\text{and}\quad\Gamma_2'=\Gamma_2\bar{\Gamma}_2.$$
		Therefore, $\Gamma_1'$ (respectively, $\Gamma_2'$) is a composite cycle in the signed directed graph $D$ of $\mathcal{P}$ consisting only of negative (respectively, positive) $2$-cycles, and since $l(\Gamma_1') + l(\Gamma_2') = m + 2$, it follows from Lemma~\ref{xnl2} that $\mathcal{P}$ does not require a unique inertia. 
	\end{itemize}
\end{proof}
The following are examples of sign patterns that satisfy conditions (i), (ii), and (iii) of the previous theorem, respectively, and therefore do not require a unique inertia.

\begin{eg}\label{xxeg12}
	Let $\mathcal{P}\in \mathcal{Q}_6$ be an irreducible combinatorially symmetric sign pattern with a $0$-diagonal, whose underlying undirected graph $G$ is  given in Fig.\ref{xnfig4.1},
	\[
	\mathcal{P} =
	\begin{bmatrix}
		0 & + & 0 & + & 0 & 0 \\
		- & 0 & + & 0 & 0 & 0 \\
		0 & + & 0 & + & 0 & 0 \\
		+ & 0 & + & 0 & + & 0 \\
		0 & 0 & 0 & + & 0 & + \\
		0 & 0 & 0 & 0 & + & 0
	\end{bmatrix}
	\]
	
	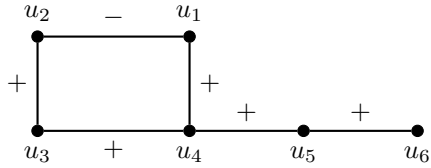
\begin{figure}[H]
		\centering
		\tikzset{node distance=1cm}
		\begin{tikzpicture}
			\tikzstyle{vertex}=[draw, circle, inner sep=0pt, minimum size=0.15cm, fill=black]
			\tikzstyle{edge}=[thick]
			\node[vertex, label=above:$u_2$](v2) at (-2,1.25) {};
			\node[vertex, label=above:$u_1$](v1) at (0,1.25) {};
			\node[vertex, label=below:$u_3$](v3) at (-2,0) {};
			\node[vertex, label=below:$u_4$](v4) at (0,0) {};
			\node[vertex, label=below:$u_5$](v5) at (1.5,0) {};
			\node[vertex, label=below:$u_6$](v6) at (3,0) {};
			
			\draw[edge] (v1)--node[above]{$-$} (v2);
			\draw[edge] (v2)--node[left]{$+$} (v3);
			\draw[edge] (v3)--node[below]{$+$} (v4);
			\draw[edge] (v4)--node[above]{$+$} (v5);
			\draw[edge] (v1)--node[right]{$+$} (v4);
			\draw[edge] (v5)--node[above]{$+$} (v6);
		\end{tikzpicture}
		\caption{Signed undirected graph $G$ of $\mathcal{P}$.}
		\label{xnfig4.1}
	\end{figure}
	Then $G$ has exactly one cycle $\mathcal{C}=u_1u_2u_3u_4$, of length $4$, which contains exactly one negative edge $\{u_1,u_2\}$. The distance from the cycle $\mathcal{C}$ to the leaf vertex $u_6$ is $2$. Thus, $\mathcal{P}$ satisfies the conditions stated in (i) of Theorem~\ref{xxth3.14}, which implies that $\mathcal{P}$ does not require a unique inertia.
\end{eg}

\begin{eg}
	Let $\mathcal{P}\in \mathcal{Q}_6$ be an irreducible combinatorially symmetric sign pattern with a $0$-diagonal, whose underlying undirected graph $G$ is  given in Fig.\ref{xnfig4.2},
	\[
	\mathcal{P} =
	\begin{bmatrix}
		0 & + & 0 & - & 0 & 0 \\
		- & 0 & + & 0 & 0 & 0 \\
		0 & - & 0 & + & 0 & 0 \\
		+ & 0 & - & 0 & + & 0 \\
		0 & 0 & 0 & + & 0 & + \\
		0 & 0 & 0 & 0 & + & 0
	\end{bmatrix}
	\]
	
	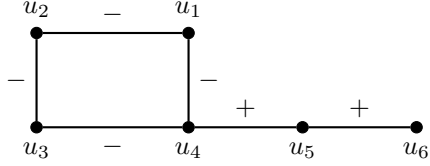
\begin{figure}[H]
		\centering
		\tikzset{node distance=1cm}
		\begin{tikzpicture}
			\tikzstyle{vertex}=[draw, circle, inner sep=0pt, minimum size=0.15cm, fill=black]
			\tikzstyle{edge}=[thick]
			\node[vertex, label=above:$u_2$](v2) at (-2,1.25) {};
			\node[vertex, label=above:$u_1$](v1) at (0,1.25) {};
			\node[vertex, label=below:$u_3$](v3) at (-2,0) {};
			\node[vertex, label=below:$u_4$](v4) at (0,0) {};
			\node[vertex, label=below:$u_5$](v5) at (1.5,0) {};
			\node[vertex, label=below:$u_6$](v6) at (3,0) {};
			
			\draw[edge] (v1)--node[above]{$-$} (v2);
			\draw[edge] (v2)--node[left]{$-$} (v3);
			\draw[edge] (v3)--node[below]{$-$} (v4);
			\draw[edge] (v4)--node[above]{$+$} (v5);
			\draw[edge] (v1)--node[right]{$-$} (v4);
			\draw[edge] (v5)--node[above]{$+$} (v6);
		\end{tikzpicture}
		\caption{Signed undirected graph $G$ of $\mathcal{P}$.}
		\label{xnfig4.2}
	\end{figure}
	Then $G$ has exactly one cycle $\mathcal{C}=u_1u_2u_3u_4$, of length $4$, with all edges negatively signed. Also, the distance from the cycle $\mathcal{C}$ to the leaf vertex $u_6$ is $2$. Thus, $\mathcal{P}$ satisfies the conditions stated in (ii) of Theorem~\ref{xxth3.14}, which implies that $\mathcal{P}$ does not require a unique inertia.
\end{eg}

\begin{eg}
	Let $\mathcal{P}\in \mathcal{Q}_6$ be an irreducible combinatorially symmetric sign pattern with a $0$-diagonal, whose underlying undirected graph $G$ is given in Example~\ref{xxeg12}.
	Then $G$ has exactly one cycle $\mathcal{C}=u_1u_2u_3u_4$, of length $4$, which contains a maximal signed path of odd length. The distance from the cycle $\mathcal{C}$ to the leaf vertex $u_6$ is $2$. Thus, $\mathcal{P}$ satisfies the conditions stated in (iii) of Theorem~\ref{xxth3.14}, which implies that $\mathcal{P}$ does not require a unique inertia.
\end{eg}

From Examples \ref{eg0.6}, \ref{egxx0.6}, we can similarly construct examples to show that if neither of the three conditions of Theorem~\ref{xxth3.14} are satisfied by a sign pattern, then we cannot conclude affirmatively in either way, that is, we can find examples of sign patterns which require a unique inertia as well as those which do not require a unique inertia.

We next turn our attention to irreducible combinatorially symmetric sign patterns with more than one cycle in the underlying graph. 

\begin{definition}
	Suppose that $\mathcal{P}$ is an irreducible combinatorially symmetric sign pattern whose underlying signed undirected graph is $G$. Then any two cycles $\mathcal{C}_1$, $\mathcal{C}_2$ in $G$ are called \textit{path-adjacent} if there exists a path $uv_1v_2\cdots v_{r}w$ in $G$ where $u\in V(\mathcal{C}_1)$, $w\in V(\mathcal{C}_2)$ and $v_1,v_2,...,v_{r}$ are not vertices of any cycle in $G$.
\end{definition}

\begin{lemma}\label{nl111}
	Let $\mathcal{P}\in \mathcal{Q}_n$ be an irreducible combinatorially symmetric sign pattern with a $0$-diagonal, whose underlying signed undirected graph is $G$ and the signed directed graph is $D$. Suppose $G$ contains at least one cycle, has no leaf, and let the distance between any two path-adjacent cycles in $G$ be odd. Let $\mathcal{C}=u_1u_2\cdots u_k$ be any cycle in $G$ where $k\geq 3$, then $D$ has a composite cycle of length $n$ containing the directed cycle $(u_1,u_2,...,u_k)$.
\end{lemma}
\begin{proof}
	We prove this result using induction in $m$, where $m$ is the number of cycles in $G$. For $m=1$, the result is trivially true.
	
	Suppose that the result holds for all $m$, $m \leq t-1$, $t\geq 2$. For $m=t$, let $u_1v_1v_2\cdots v_{2r}w_1$, $r\in \mathbb{N}$, be a minimum distance path in $G$ from the cycle $\mathcal{C}$ to some cycle $\mathcal{C}_1$, path-adjacent to $\mathcal{C}$. We denote this path by $u_1P_1w_1$, where $P_1=\{v_1,v_2\}\{v_2,v_3\}\cdots \{v_{2r-1},v_{2r}\}$ with $u_1\in V(\mathcal{C})$, $w_1\in V(\mathcal{C}_1)$ and $v_1,v_2,...,v_{2r}\notin V(\mathcal{C}),V(\mathcal{C}_1)$. Since the distance between any two path-adjacent cycles is odd and $G$ does not contain any leaf, $\deg(v_i)=2$ in $G$ for all $i=1,2,...,2r$. In addition, $\Gamma_1=(v_1,v_2)(v_3,v_4)\cdots(v_{2r-1},v_{2r})$ forms a composite cycle in $D$ containing all vertices of $P_1$.
	
	\begin{figure}[H]
		\tikzset{node distance=1cm}
		\centering
		\begin{tikzpicture}
			\tikzstyle{vertex}=[draw, circle, inner sep=0pt, minimum size=.15cm, fill=black]
			\tikzstyle{edge}=[,thick]
			\node[vertex,label=below:$~u_1$](v1)at(0,0){};
			\node[vertex,label=below:$v_1$](v7)at(1,0){};
			\node[vertex,label=below:$v_2$](v8)at(2,0){};
			\node[vertex,label=below:$v_{2r-1}$](v9)at(3,0){};
			\node[vertex,label=below:$v_{2r}$](v10)at(4,0){};
			\node[vertex, label=below:$w_1~$](v11)at(5,0){};
			\node[vertex](v13)at(-0.75,1){};
			\node[vertex](v14)at(-0.75,-1){};
			\node[vertex](v15)at(5.75,1){};
			\node[vertex](v16)at(5.75,-1){};
			
			\draw[edge](0,1.5) to [bend left=35] (0,-1.9);
			\draw[edge](5,1.5) to [bend right=35] (5,-1.9);
			\draw[edge](v1)--node[]{$ $}(v7);
			\draw[edge](v7)--node[]{$ $}(v8);
			\draw[thick, dotted](v8)--node[]{$ $}(v9);
			\draw[edge](v9)--node[]{$ $}(v10);
			\draw[edge](v10)--node[]{$ $}(v11);	
			\draw[thick, dotted](v1)--node[]{$ $}(v13);
			\draw[thick, dotted](v1)--node[]{$ $}(v14);
			\draw[thick, dotted](v11)--node[]{$ $}(v15);
			\draw[thick, dotted](v11)--node[]{$ $}(v16);
		\end{tikzpicture}
		\caption{$G$}
	\end{figure}

	Similarly, in every possible minimum distance path $u_iP_iw_i$, from $\mathcal{C}$ to $\mathcal{C}_i$ in $G$ the degree of each vertex in $P_i$ is exactly equal to $2$ in $G$, and there exists a composite cycle $\Gamma_i$ in $D$ containing all vertices of $P_i$. 
	
	Consider $G'=G\setminus (V(\mathcal{C})\cup V(P_1)\cup...\cup V(P_s))$ and let the corresponding signed directed graph be $D'$. Then each of the connected components of $G'$ has at least one cycle, no leaf, such that the distance between any two path-adjacent cycles is odd. Also, the number of cycles in each of the components is at most $t-1$. So, by the induction hypothesis, $D'$ has a composite cycle $\Gamma'$ of length equal to the order of $G'$. Let $\Gamma=\Gamma' \Gamma_1\Gamma_2 \cdots  \Gamma_s (u_1u_2\cdots u_k).$ Then $\Gamma$ is a composite cycle of length $n$ in $D$ containing $(u_1,u_2,...,u_k)$.
\end{proof}

If the distance between two path-adjacent cycles in $G$ is even, then the conclusion of the above lemma does not hold good, as the following examples show. 
\begin{eg}
	Let $\mathcal{P}\in \mathcal{Q}_9$ be an irreducible combinatorially symmetric sign pattern with a $0$-diagonal, whose underlying undirected graph $G$ is given in Fig.\ref{xxfig1}.
	\begin{figure}[H]
		\tikzset{node distance=1cm}
		\centering
		\begin{tikzpicture}
			\tikzstyle{vertex}=[draw, circle, inner sep=0pt, minimum size=.15cm, fill=black]
			\tikzstyle{edge}=[,thick]
			\node[vertex,label=above:$u_2$, ](v1)at(-2,1){};
			\node[vertex,label=above:$u_1$](v8)at(0,1){};
			\node[vertex,label=below:$u_3$](v2)at(-2,0){};
			\node[vertex,label=below:$u_4$](v3)at(0,0){};
			\node[vertex,label=below:$u_5$](v4)at(1,0){};
			\node[vertex,label=below:$u_6$](v5)at(2,0){};
			\node[vertex,label=below:$u_7$](v6)at(4,0){};
			\node[vertex,label=above:$u_8$](v7)at(4,1){};
			\node[vertex,label=above:$u_9$](v9)at(2,1){};
			
			\draw[edge](v1)--node[]{$ $}(v2);
			\draw[edge](v1)--node[]{$ $}(v8);
			\draw[edge](v2)--node[]{$ $}(v3);
			\draw[edge](v3)--node[]{$ $}(v8);
			\draw[edge](v3)--node[]{$ $}(v4);
			\draw[edge](v4)--node[]{$ $}(v5);
			\draw[edge](v5)--node[]{$ $}(v6);
			\draw[edge](v9)--node[]{$ $}(v7);
			\draw[edge](v7)--node[]{$ $}(v6);
			\draw[edge](v5)--node[]{$ $}(v9);
		\end{tikzpicture}
		\caption{$G$}
		\label{xxfig1}
	\end{figure} Then $G$ contains cycles $\mathcal{C}_1=u_1u_2u_3u_4$, $\mathcal{C}_2=u_6u_7u_8u_9$ of length $4$. Let $D$ be the signed directed graph corresponding to $G$. Then the maximum length of the composite cycles in $D$ is $8$; however, the order of $D$ is $9$.      
\end{eg}

\begin{eg}
	Let $\mathcal{P}\in \mathcal{Q}_8$ be an irreducible combinatorially symmetric sign pattern with a $0$-diagonal, whose underlying undirected graph $G$ is given in Fig.\ref{xxfig1.11}.
	\begin{figure}[H]
		\tikzset{node distance=1cm}
		\centering
		\begin{tikzpicture}
			\tikzstyle{vertex}=[draw, circle, inner sep=0pt, minimum size=.15cm, fill=black]
			\tikzstyle{edge}=[,thick]
			\node[vertex,label=above:$u_2$, ](v1)at(-2,1){};
			\node[vertex,label=above:$u_1$](v8)at(0,1){};
			\node[vertex,label=below:$u_3$](v2)at(-2,0){};
			\node[vertex,label=below:$u_4$](v3)at(0,0){};
			\node[vertex,label=below:$u_5$](v4)at(1,0){};
			\node[vertex,label=below:$u_6$](v5)at(2,0){};
			\node[vertex,label=below:$u_7$](v6)at(4,0){};
			\node[vertex,label=above:$u_8$](v7)at(3,1){};
			
			\draw[edge](v1)--node[]{$ $}(v2);
			\draw[edge](v1)--node[]{$ $}(v8);
			\draw[edge](v2)--node[]{$ $}(v3);
			\draw[edge](v3)--node[]{$ $}(v8);
			\draw[edge](v3)--node[]{$ $}(v4);
			\draw[edge](v4)--node[]{$ $}(v5);
			\draw[edge](v5)--node[]{$ $}(v6);
			\draw[edge](v5)--node[]{$ $}(v7);
			\draw[edge](v7)--node[]{$ $}(v6);
		\end{tikzpicture}
		\caption{$G$}
		\label{xxfig1.11}
	\end{figure} Then $\mathcal{C}_1=u_1u_2u_3u_4$, $\mathcal{C}_2=u_6u_7u_8$ are two cycles in $G$. Let $D$ be the signed directed graph corresponding to $G$. Although $D$ has a composite cycle of length $8$, but $D$ does not have a composite cycle of length $8$ containing the cycle $(u_6u_7u_8)$.      
\end{eg}

\begin{theorem}\label{xxth3.22}
	Let $\mathcal{P}\in \mathcal{Q}_n$ be an irreducible combinatorially symmetric sign pattern with a $0$-diagonal, whose underlying signed undirected $G$ has at least one cycle, has no leaf, and let the distance between any two path-adjacent cycles in $G$ be odd. Then $\mathcal{P}$ does not require a unique inertia if any one of the following holds:
	\begin{itemize}
		\item[i.] $G$ has a cycle $\mathcal{C}$ with an odd number of negative edges.
		\item[ii.] All cycles in $G$ have even length and $G$ has a cycle $\mathcal{C}$ with all edges negatively signed.
		\item[iii.] All cycles in $G$ have even length and $G$ has a cycle $\mathcal{C}$ containing a maximal signed path of odd length.
	\end{itemize}
\end{theorem}
\begin{proof} If $G$ has exactly one cycle, then it follows from Theorem~\ref{xxth3} that $\mathcal{P}$ does not require a unique inertia. Suppose that $G$ has more than one cycle, and let $\mathcal{C} = u_1u_2\cdots u_k$ be a cycle in $G$. Let $D$ be the signed directed graph corresponding to $G$. 
	Consider a composite cycle $\Gamma = \beta_1 \beta_2 \cdots \beta_t$ of maximum length 
	in the subgraph $D \setminus V(\mathcal{C})$, where each $\beta_r$ is a simple cycle in $D$ 
	for $1 \leq r \leq t$, and let $l(\Gamma) = l$. Then, by Lemma~\ref{nl111}, we have $n = k + l$.  
	\begin{itemize}
		\item[i.] If $\mathcal{C}$ has an odd number of negative edges, then 
		$\gamma_1=(u_1u_2\cdots u_k)$ and $\gamma_2=(u_ku_{k-1}\cdots u_1)$ are two cycles in $D$ of length $k$ with $\sgn(\gamma_1)=-\sgn(\gamma_2)$. 
		Therefore, 
		$$
		\Gamma_1 = \gamma_1\Gamma \quad\text{and}\quad \Gamma_2 = \gamma_2\Gamma
		$$
		are two composite cycles in $D$ of length $n$ satisfying $\sgn(\Gamma_1)=-\sgn(\Gamma_2)$. 
		By Lemma~\ref{xnl1}, $\mathcal{P}$ does not require a unique inertia.

		\item[ii.] Similarly as in the proof of Theorem~\ref{xxth3}(ii), $D$ has two composite cycles
		$\Gamma_1,\Gamma_2$ of length $k$ and two real matrix $B_{\Gamma_1}(\epsilon),B_{\Gamma_2}(\epsilon) \in \mathcal{Q}(\mathcal{P}[\{u_1,u_2,...,u_k\}])$ with $\In(B_{\Gamma_1}(\epsilon))\neq \In(B_{\Gamma_2}(\epsilon))$. Also, from the proof of Theorem~\ref{xxth3.14}(ii), the eigenvalues of both $B_{\Gamma_1}(\epsilon),B_{\Gamma_2}(\epsilon)$ are distinct.  Let,
		$$M>\max\{1,|\lambda|: \lambda~ \text{is an eigenvalue of either} ~B_{\Gamma_1}(\epsilon)~\text{or}~B_{\Gamma_2}(\epsilon)\}.$$ 
		Consider $$\Gamma_1'=\Gamma_1\Gamma\quad\text{and}\quad\Gamma_2'=\Gamma_2\Gamma$$
		and define the real matrices $B_{\Gamma_r'}(\epsilon)=[b_{\Gamma_r'}(\epsilon)_{ij}]\in \mathcal{Q}(\mathcal{P})$ for $r=1,2$, as
		$$|b_{\Gamma_r'}(\epsilon)_{ij}|= \begin{cases}
			|b_{\Gamma_r}(\epsilon)_{ij}| & \text{if}\ p_{ij}~ \text{is in } \Gamma_r \\
			M^p & \text{if}\ p_{ij}~ \text{is in}~ \beta_p,~\text{for}~p=1,2,...,t \\
			\epsilon & \text{if}~ p_{ij}\neq 0~\text{and not in}~\Gamma_r, \Gamma\\
			0 & \text{elsewhere.}
		\end{cases}$$
		For $\epsilon>0$ sufficiently small, the eigenvalues of $B_{\Gamma_r'}(\epsilon)$ remain close to the eigenvalues of $B_{\Gamma_r}(\epsilon)$  for $r=1,2$, together with the $l(\beta_p)$-th complex roots of $M^{pl(\beta_p)}$, $-M^{pl(\beta_p)}$, depending on the sign of $\beta_p$, for $p=1,3,...,t$. Since $\In(B_{\Gamma_1}(\epsilon))\neq \In(B_{\Gamma_2}(\epsilon))$ and $D$ contains only cycles of even length, by Lemma~\ref{lem1}, $\In(B_{\Gamma_1'}(\epsilon))\neq \In(B_{\Gamma_2'}(\epsilon))$. Therefore, $\mathcal{P}$ does not require a unique inertia.

		\item[iii.] If $l(\mathcal{C})$ is even and $\mathcal{C}$ contains a maximal signed path of odd length, it follows similarly to the proof of Theorem~\ref{xxth3}(iii) that there exist two matchings $\mathcal{M}_1=\{\alpha_{i_1},\alpha_{i_2},..., \alpha_{i_{k_1}}\}$ and $\mathcal{M}_2=\{\alpha_{j_1},\alpha_{j_2},..., \alpha_{j_{k_2}}\}$ such that $\mathcal{M}_1$ contains only negative edges and $\mathcal{M}_2$ contains only positive edges, satisfying the relation $2k_1 + 2k_2 = k + 2$, where $\alpha_t$ for $t=1,2,...,k$ are the edges of $\mathcal{C}$. Let $\Gamma_1$, $\Gamma_2$ denote the composite cycles in $D$ corresponding to $\mathcal{M}_1$, $\mathcal{M}_2$, respectively. Consider
		$$\Gamma_1'=\Gamma_1\Gamma\quad\text{and}\quad\Gamma_2'=\Gamma_2\Gamma.$$ Define real matrices $B_{\Gamma_1'}(\epsilon)=[b_{\Gamma_1'}(\epsilon)_{ij}]\in \mathcal{Q}(\mathcal{P})$ as
		$$|b_{\Gamma_1'}(\epsilon)_{ij}|= \begin{cases}
			10^p & \text{if}\ p_{ij}~ \text{is}~ \text{in}~ \alpha_p,~\text{for}~p=i_1,i_2,...,i_{k_1} \\
			11^p & \text{if}\ p_{ij}~ \text{is}~ \text{in}~ \beta_p,~\text{for}~p=1,2,...,t \\
			\epsilon & \text{if}~ p_{ij}\neq 0~\text{and not in both}~ \Gamma_1,~ \Gamma\\
			0 & \text{elsewhere.}
		\end{cases}$$
		
		Since $D$ contains only cycles of even length, the characteristic polynomial $\Ch_{B_{\Gamma_r'}(\epsilon)}(x)$ consists of only even powers of $x$. Hence, by Lemma~\ref{lem1}, for sufficiently small $\epsilon > 0$, we have $$i_0(B_{\Gamma_1'}(\epsilon))\geq 2k_1+i_0(B_{\Gamma}(\epsilon)).$$

		Similarly, define the real matrix $B_{\Gamma_2'} (\epsilon)\in \mathcal{Q}(\mathcal{P})$, for sufficiently small $\epsilon > 0$, we have $$n-i_0(B_{\Gamma_2'}(\epsilon))\geq 2k_2+(l-i_0(B_{\Gamma}(\epsilon))),$$
		where $B_{\Gamma}(\epsilon)\in \mathcal{Q}(\mathcal{P}[\{u_1,u_2,...,u_k\}^c])$ is a submatrix of both $B_{\Gamma_1'}(\epsilon)$, $B_{\Gamma_2'}(\epsilon)$.
		Therefore,
		$$i_0(B_{\Gamma_1'}(\epsilon))+(n-i_0(B_{\Gamma_2'}(\epsilon)))\geq (2k_1+i_0(B_{\Gamma}(\epsilon)))+(2k_2+(l-i_0(B_{\Gamma}(\epsilon))))= n+2.$$ Hence, by Theorem \ref{lem2}, $\mathcal{P}$ does not require a unique inertia.  \end{itemize}
\end{proof}

The following are examples of sign patterns that satisfy conditions (i), (ii), and (iii) of the previous theorem, respectively, and therefore do not require a unique inertia.

\begin{eg}
	Let $\mathcal{P}\in \mathcal{Q}_8$ be an irreducible combinatorially symmetric sign pattern with a $0$-diagonal, whose underlying undirected graph $G$ is  given in Fig.\ref{xnfig5.1},
	\[
	\mathcal{P} =
	\begin{bmatrix}
		0 & + & + & 0 & 0 & 0 & 0 & 0 \\
		- & 0 & + & 0 & 0 & 0 & 0 & 0 \\
		+ & + & 0 & + & 0 & 0 & 0 & 0 \\
		0 & 0 & + & 0 & + & 0 & 0 & 0 \\
		0 & 0 & 0 & - & 0 & + & 0 & 0 \\
		0 & 0 & 0 & 0 & + & 0 & + & + \\
		0 & 0 & 0 & 0 & 0 & + & 0 & + \\
		0 & 0 & 0 & 0 & 0 & + & + & 0
	\end{bmatrix}
	\]
	
	\begin{figure}[H]
		\centering
		\tikzset{node distance=1cm}
		\begin{tikzpicture}
			\tikzstyle{vertex}=[draw, circle, inner sep=0pt, minimum size=0.15cm, fill=black]
			\tikzstyle{edge}=[thick]
			\node[vertex, label=above:$u_1$](v1) at (-2.5,1.25) {};
			\node[vertex, label=below:$u_2$](v2) at (-3.5,0) {};
			\node[vertex, label=below:$u_3$](v3) at (-1.5,0) {};
			\node[vertex, label=below:$u_4$](v4) at (0,0) {};
			\node[vertex, label=below:$u_5$](v5) at (1.5,0) {};
			\node[vertex, label=below:$u_6$](v6) at (3,0) {};
			\node[vertex, label=below:$u_7$](v7) at (5,0) {};
			\node[vertex, label=above:$u_8$](v8) at (4,1.25) {};
			
			\draw[edge] (v1)--node[left]{$-$} (v2);
			\draw[edge] (v2)--node[below]{$+$} (v3);
			\draw[edge] (v3)--node[right]{$+$} (v1);
			\draw[edge] (v3)--node[above]{$+$} (v4);
			\draw[edge] (v4)--node[above]{$-$} (v5);
			\draw[edge] (v5)--node[above]{$+$} (v6);
			\draw[edge] (v7)--node[below]{$+$} (v6);
			\draw[edge] (v7)--node[right]{$+$} (v8);
			\draw[edge] (v8)--node[left]{$+$} (v6);
		\end{tikzpicture}
		\caption{Signed undirected graph $G$ of $\mathcal{P}$.}
		\label{xnfig5.1}
	\end{figure}
	Then $G$ has a cycle $\mathcal{C}=u_1u_2u_3$ with an odd number of negative edges, and the distance from the cycle $\mathcal{C}$ to $\mathcal{C}_1=u_6u_7u_8$ is odd. Thus, $\mathcal{P}$ satisfies the conditions stated in (i) of Theorem~\ref{xxth3.22}, which implies that $\mathcal{P}$ does not require a unique inertia.
\end{eg}

\begin{eg}
	Let $\mathcal{P}\in \mathcal{Q}_8$ be an irreducible combinatorially symmetric sign pattern with a $0$-diagonal, whose underlying undirected graph $G$ is  given in Fig.\ref{xnfig6.1},
	\[
	\mathcal{P} =
	\begin{bmatrix}
		0 & + & 0 & - & 0 & 0 & 0 & 0 \\
		- & 0 & + & 0 & 0 & 0 & 0 & 0 \\
		0 & - & 0 & + & 0 & 0 & 0 & 0 \\
		+ & 0 & - & 0 & + & 0 & 0 & 0 \\
		0 & 0 & 0 & + & 0 & + & 0 & + \\
		0 & 0 & 0 & 0 & + & 0 & + & 0 \\
		0 & 0 & 0 & 0 & 0 & + & 0 & + \\
		0 & 0 & 0 & 0 & + & 0 & + & 0
	\end{bmatrix}
	\]
	
	\begin{figure}[H]
		\centering
		\tikzset{node distance=1cm}
		\begin{tikzpicture}
			\tikzstyle{vertex}=[draw, circle, inner sep=0pt, minimum size=0.15cm, fill=black]
			\tikzstyle{edge}=[thick]
			\node[vertex, label=above:$u_1$](v1) at (0,1.25) {};
			\node[vertex, label=above:$u_2$](v2) at (-2,1.25) {};
			\node[vertex, label=below:$u_3$](v3) at (-2,0) {};
			\node[vertex, label=below:$u_4$](v4) at (0,0) {};
			\node[vertex, label=below:$u_5$](v5) at (1.5,0) {};
			\node[vertex, label=below:$u_6$](v6) at (3.5,0) {};
			\node[vertex, label=above:$u_7$](v7) at (3.5,1.25) {};
			\node[vertex, label=above:$u_8$](v8) at (1.5,1.25) {};
			
			\draw[edge] (v1)--node[above]{$-$} (v2);
			\draw[edge] (v2)--node[left]{$-$} (v3);
			\draw[edge] (v3)--node[below]{$-$} (v4);
			\draw[edge] (v4)--node[right]{$-$} (v1);
			\draw[edge] (v4)--node[below]{$+$} (v5);
			\draw[edge] (v5)--node[below]{$+$} (v6);
			\draw[edge] (v7)--node[right]{$+$} (v6);
			\draw[edge] (v7)--node[above]{$+$} (v8);
			\draw[edge] (v8)--node[left]{$+$} (v5);
		\end{tikzpicture}
		\caption{Signed undirected graph $G$ of $\mathcal{P}$.}
		\label{xnfig6.1}
	\end{figure}
	Then $G$ has a cycle $\mathcal{C}=u_1u_2u_3u_4$ with all edges are negatively signed, and the distance from the cycle $\mathcal{C}$ to $\mathcal{C}_1=u_5u_6u_7u_8$ is odd. Thus, $\mathcal{P}$ satisfies the conditions stated in (ii) of Theorem~\ref{xxth3.22}, which implies that $\mathcal{P}$ does not require a unique inertia.
\end{eg}

\begin{eg}
	Let $\mathcal{P}\in \mathcal{Q}_8$ be an irreducible combinatorially symmetric sign pattern with a $0$-diagonal, whose underlying undirected graph $G$ is  given in Fig.\ref{xnfig6.2},
	\[
	\mathcal{P} =
	\begin{bmatrix}
		0 & + & 0 & + & 0 & 0 & 0 & 0 \\
		- & 0 & + & 0 & 0 & 0 & 0 & 0 \\
		0 & + & 0 & + & 0 & 0 & 0 & 0 \\
		+ & 0 & + & 0 & + & 0 & 0 & 0 \\
		0 & 0 & 0 & + & 0 & + & 0 & + \\
		0 & 0 & 0 & 0 & + & 0 & + & 0 \\
		0 & 0 & 0 & 0 & 0 & + & 0 & + \\
		0 & 0 & 0 & 0 & + & 0 & + & 0
	\end{bmatrix}
	\]
	
	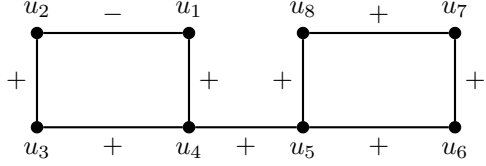
\begin{figure}[H]
		\centering
		\tikzset{node distance=1cm}
		\begin{tikzpicture}
			\tikzstyle{vertex}=[draw, circle, inner sep=0pt, minimum size=0.15cm, fill=black]
			\tikzstyle{edge}=[thick]
			\node[vertex, label=above:$u_1$](v1) at (0,1.25) {};
			\node[vertex, label=above:$u_2$](v2) at (-2,1.25) {};
			\node[vertex, label=below:$u_3$](v3) at (-2,0) {};
			\node[vertex, label=below:$u_4$](v4) at (0,0) {};
			\node[vertex, label=below:$u_5$](v5) at (1.5,0) {};
			\node[vertex, label=below:$u_6$](v6) at (3.5,0) {};
			\node[vertex, label=above:$u_7$](v7) at (3.5,1.25) {};
			\node[vertex, label=above:$u_8$](v8) at (1.5,1.25) {};
			
			\draw[edge] (v1)--node[above]{$-$} (v2);
			\draw[edge] (v2)--node[left]{$+$} (v3);
			\draw[edge] (v3)--node[below]{$+$} (v4);
			\draw[edge] (v4)--node[right]{$+$} (v1);
			\draw[edge] (v4)--node[below]{$+$} (v5);
			\draw[edge] (v5)--node[below]{$+$} (v6);
			\draw[edge] (v7)--node[right]{$+$} (v6);
			\draw[edge] (v7)--node[above]{$+$} (v8);
			\draw[edge] (v8)--node[left]{$+$} (v5);
		\end{tikzpicture}
		\caption{Signed undirected graph $G$ of $\mathcal{P}$.}
		\label{xnfig6.2}
	\end{figure}
	Then $G$ has a cycle $\mathcal{C}=u_1u_2u_3u_4$ with a maximal signed path of odd length, and the distance from the cycle $\mathcal{C}$ to $\mathcal{C}_1=u_5u_6u_7u_8$ is odd. Thus, $\mathcal{P}$ satisfies the conditions stated in (iii) of Theorem~\ref{xxth3.22}, which implies that $\mathcal{P}$ does not require a unique inertia.
\end{eg}

If neither of the three conditions of Theorem~\ref{xxth3.22} are satisfied by a sign pattern, from Examples \ref{eg0.6}, \ref{egxx0.6}, we can similarly construct examples of sign patterns which require a unique inertia as well as those which do not require a unique inertia.

\section{Concluding remarks}\label{s5}

Some of the problems that leave the scope for future research are related to investigating the relationship between consistency and requiring a unique inertia for tridiagonal sign patterns with a $0$-diagonal, i.e., whether such patterns require a unique inertia if and only if they are consistent.

We believe that if $\mathcal{P}\in \mathcal{Q}_n$ is an irreducible tridiagonal sign pattern with a $0$-diagonal such that the underlying signed undirected graph has a maximal signed path of odd length that does not include any leaf of $G$, then by arguments similar to those in Theorems~\ref{nth3.4} and \ref{nth3.18}, it can be shown that $\mathcal{P}$ does not require a unique inertia.
Also, if the only maximal signed odd length path always includes a leaf for such patterns, or if it does not contain any maximal signed odd length path, then whether it implies that the pattern requires a unique inertia needs to be investigated.

Another open question is, if a path sign pattern a with $0$-diagonal allow repeated eigenvalues with zero real part, then can we say that it does not require a unique inertia? We have only partially answered this question, and we could show that if it allows a repeated zero eigenvalue, then it does not require a unique inertia.
Also, we have noticed that although in this paper we focus only on combinatorially symmetric sign patterns with a $0$-diagonal, some of the results can be clearly extended to their super patterns by following almost the same proof. So identifying the super patterns of those studied in this paper, requiring a unique inertia, will be an interesting topic for future research.

\section{Acknowledgement} The research work of Partha Rana was supported by the Council of Scientific and Industrial Research (CSIR), India (File Number 09/731(0186)/2021-EMR-I).

\end{document}